\documentclass[10pt]{amsproc}
\usepackage{fullpage}
\usepackage{graphicx}
\usepackage{amssymb}
\usepackage{amsfonts}
\usepackage{amsthm}
\usepackage{amsmath}
\usepackage[shortalphabetic]{amsrefs}
\usepackage{xypic}
\usepackage{eucal}
\usepackage{pdfsync}
\usepackage{hyperref}
\usepackage[margin=1.25in]{geometry}

\def\ra{\rightarrow}
\def\Z{{\mathbb Z}}

\def\R{{\mathbb R}}
\def\C{{\mathbb C}}

\def\F{{\mathbb F}}
\def\K{{\mathbb K}}

\def\cc{\mathcal{C}}
\def\dd{\mathcal{D}}
\def\bb{\mathcal{B}}

\def\b{\beta}
\def\r#1{\mathrm{#1}}

\def\mc#1{\mathcal{#1}}
\def\mc#1{\mathcal{#1}}

\def\ainf{A_\infty}

\def\z2{\Z / 2\Z}

\def\lra{\longrightarrow}
\def\ra{\rightarrow}

\def\ob{\mathrm{ob\ }}

\newtheorem{lem}{Lemma}
\newtheorem{prop}[lem]{Proposition}
\newtheorem{thm}[lem]{Theorem}
\newtheorem{cor}[lem]{Corollary}
\newtheorem{defn}[lem]{Definition}

\theoremstyle{remark}
\newtheorem{rem}[lem]{Remark}

\numberwithin{equation}{section}
\begin{document}

\begin{abstract}
Suppose one has found a non-empty sub-category $\mathcal{A}$ of the Fukaya category of a compact Calabi-Yau manifold $X$ which is {\em homologically smooth} in the sense of non-commutative geometry, an algebraic condition intrinsic to $\mathcal{A}$.  Then, we show $\mathcal{A}$ split-generates the Fukaya category and moreoever, that our hypothesis implies (and is therefore equivalent to the assertion that) $\mathcal{A}$ satisfies Abouzaid's geometric generation criterion \cite{Abouzaid:2010kx}. An immediate consequence of earlier work \cites{ganatra1_arxiv, GPS1:2015, GPS2:2015} is that the open-closed and closed-open maps, relating quantum cohomology to the Hochschild invariants of the Fukaya category, are also isomorphisms. 
Our result continues to hold when $c_1(X) \neq 0$ (for instance, when $X$ is monotone Fano), under a further hypothesis: the 0th Hochschild cohomology of $\mathcal{A}$ $\r{HH}^0(\mathcal{A})$ should have sufficiently large rank: $\mathrm{rk}\ \r{HH}^0(\mathcal{A}) \geq \mathrm{rk}\ \mathrm{QH}^0(X)$.  Our proof depends only on formal axiomatizable structures of Fukaya categories and open-closed maps, the most recent and crucial of which,  compatibility of the open-closed map with pairings, was observed independently in ongoing joint work of the author with Perutz and Sheridan \cite{GPS2:2015} and by Abouzaid-Fukaya-Oh-Ohta-Ono \cite{afooo}; a proof in the simplest technical settings appears here in an Appendix.
Because categories Morita equivalent to categories of coherent sheaves or matrix factorizations are homologically smooth, our result applies to resolve the split-generation question in homological mirror symmetry for compact symplectic manifolds (generalizing a result of Perutz-Sheridan \cite{Perutz:2015aa} proven in the case $c_1(X) = 0$): 
we now know that any embedding of (a split-generating subcategory of) coherent sheaves or matrix factorizations into the split-closed derived Fukaya category is automatically a Morita equivalence when it has large enough $\mathrm{HH}^0$ (which it always does if $c_1(X)=0$).
\end{abstract}

\def\oc{\mc{OC}}
\def\co{\mc{CO}}
\def\b{\mc{B}}
\def\z{ { \mathbb Z}}
\def\hhmf{\mathrm{R} \Gamma (\wedge^\bullet T_Y, [W - y, \cdot])}
\title{Automatically generating Fukaya categories and computing quantum cohomology}
\author{Sheel Ganatra}
\thanks{The author was partly supported by an NSF postdoctoral fellowship.}
\maketitle
\section{Introduction}
\subsection{Main Result}
A recent result of Perutz and Sheridan
\cite{Perutz:2015aa} shows that the (split-closed) derived Fukaya category of a
compact Calabi-Yau manifold $X$, $perf(\mc{F}(X))$, satisfies a certain rigidity property expressed
via mirror symmetry: any embedding $Coh(Y) \hookrightarrow perf(\mc{F}(X))$ of
the (dg enhancement of the) derived category of coherent sheaves $Coh(Y)$ of a Calabi-Yau family $Y$ of
varieties with maximally unipotent monodromy is automatically a
quasi-equivalence, meaning its image split-generates the Fukaya category.
From a concrete geometric perspective, any Lagrangian in
$X$ with non-vanishing Floer cohomology must intersect the Lagrangians in the
image of any generating collection of sheaves.

The goal of our work is to establish more general rigidity, or {\em automatic
split-generation}, properties for the Fukaya category of a closed symplectic
manifold $X$, independent of mirror symmetry or the Calabi-Yau hypotheses (we
note that the generation argument in \cite{Perutz:2015aa}, besides only
applying to images of embeddings from such categories $Coh(Y)$,
crucially requires $c_1(X) = 0$).  Our argument, which is relatively short and
completely different from the methods of \cite{Perutz:2015aa}, still has
the strongest consequence when $c_1(X) = 0$:
\begin{thm}
    \label{thm:mainCY}
   Let $X$ be a symplectic manifold with $c_1(X) = 0$ and let $\mc{A} \subset
   \mc{F}(X)$ be a non-empty full subcategory of its ($\Z$-graded) Fukaya
   category over a field $\K$.
    Suppose that $\mc{A}$ is {\em homologically smooth}.
    Then $\mc{A}$ split-generates $\mc{F}(X)$.
\end{thm}
{\em (Homological) smoothness}, a concept introduced by Kontsevich
\cites{Kontsevich:1999aa, Kontsevich:2009ab} and recalled in Definition
\ref{def:smooth}, is an algebraic finiteness condition amounting to possessing
an algebraic resolution of the diagonal in the sense of Beilinson
\cite{Beuilinson:1978bh}. Smoothness is a Morita invariant notion, and
generalizes a finiteness property possessed by the coordinate ring of a smooth
complex affine variety.  The key here is that smoothness of $\mc{A} \subset
\mc{F}(X)$ is a condition manifestly {\em intrinsic} to $\mc{A}$ as an ($\ainf$
or dg) category, not requiring further study of the geometry of $X$.
Algebro-geometric arguments imply a given $\mc{A}$ is automatically
smooth whenever it is Morita equivalent to (e.g., equivalent to a
split-generating subcategory of) a category of coherent sheaves or graded
matrix factorizations on some (separated finite type) scheme $Y$ over (perfect)
$\K$; see \S \ref{sec:smoothness}.  In particular,
it follows that Theorem \ref{thm:mainCY}
recovers the generation result of \cite{Perutz:2015aa} (see Corollary
\ref{coreHMSCY}).\footnote{Unlike \cite{Perutz:2015aa}, we do not not need
to impose an a priori `maximal unipotency' hypotheses on the candidate mirror $Y$, though
this will a posteriori hold if working with the category  $Coh(Y)$ for $Y$ a Calabi-Yau
variety over a Novikov field; see Remark \ref{restrictivehypothesis}.} 
\begin{rem}[On the necessity of smoothness for split-generation, compare \cite{ganatra1_arxiv}] \label{smoothnessnecessary}
    We prove Theorem \ref{thm:mainCY} and subsequent variants by showing that
    under 
    $\mathcal{A}$ automatically satisfies Abouzaid's
    well-known geometric split-generation criterion for Fukaya categories \cite{Abouzaid:2010kx}, reviewed in
    \S \ref{intro:oc}.  In fact, in order to appeal to Abouzaid's criterion our
    hypotheses are minimal (and in particular necessary): see Theorem
    \ref{thm:smoothness} for a shortened proof (of a result from \cite{ganatra1_arxiv})
    that if $\mc{A} \subset \mc{F}(X)$ satisfies Abouzaid's criterion, it must be smooth and non-empty.  

In particular, for any compact symplectic manifold $X$, 
Theorem \ref{thm:mainCY} along with its variants in non-Calabi-Yau cases (Theorem \ref{thm:mainFano} and Remark
\ref{rem:generalFukayaresult}) completely reduces the verification of
Abouzaid's geometric criterion for $\mc{A}$ (a priori involving new computations of
pseudoholomorphic discs beyond those appearing in $\mc{A}$) to an intrinsic algebraic criterion
(requiring only computation of $\mc{A}$ itself, and when $c_1(X) \neq 0$ some
information about cohomology of the ambient $X$). 
\end{rem}

The proof of Theorem \ref{thm:mainCY} is a very short consequence of
structural properties of Fukaya categories and open-closed maps (none of which
are new to this work), and is therefore valid for any construction of the
Fukaya category that, along with the quantum cohomology $\r{QH}^*(X)$,
satisfies axiomatics detailed in \S \ref{sec:geometry} - \ref{sec:openclosed}.
In particular, the same techniques apply to yield new automatic
split-generation results in non-Calabi-Yau settings, as we will now describe.
For simplicity, we will keep the focus on settings in which a classical
construction of Fukaya categories (along with verification of all axiomatics)
is available, and defer the most general statement of our result to Remark
\ref{rem:generalFukayaresult}; the proof is entirely the same (assuming the
relevant axiomatics).  For instance, let us suppose $X$ is monotone (e.g., a Fano
variety equipped with its monotone symplectic form). It is well known that the
monotone Fukaya category of $X$ splits into summands, denoted $\mc{F}_w$,
indexed by eigenvalues $w$ of the quantum multiplication operator $c_1(X) \star
- $; see \S\ref{monotoneFuk}.  Denoting
by $\r{QH}^*(X)_w$ the corresponding generalized eigenspace of
$\r{QH}^*(X)$, our main result in this case is:  
\begin{thm}\label{thm:mainFano}
    Let $\mc{A} \subset \mc{F}_w$ be a full subcategory. 
    Suppose that
    \begin{enumerate}
        \item $\mc{A}$ is {\em homologically smooth}; and 
        \item\label{rankinequality} there is a rank
    inequality $\mathrm{rk}\ \r{HH}^{0}(\mc{A}) \geq \mathrm{rk}\ \r{QH}^0(X)_w$, where $\r{HH}^*(\mc{A})$ denotes the {\em Hochschild cohomology} of $\mc{A}$ (see \S \ref{sec:algebra}).
    \end{enumerate}
     Then $\mc{A}$ split-generates $\mc{F}_w$ (and the
     inequality in (\ref{rankinequality}) is an equality).
\end{thm}
\begin{rem}
    The rank inequality (\ref{rankinequality}) 
    appearing in Theorem \ref{thm:mainFano} is required when $c_1(X) = 0$ as
    well, but there it is automatically satisfied: the presence of a
    $\Z$-grading implies $\r{rk}\ QH^0(X) = \r{rk}\ H^0(X) = 1$, and
    $\mathrm{rk}\ \r{HH}^{0}(\mc{A}) \geq 1$ whenever $\mc{A}$ is non-empty
    (see Lemma \ref{lem:HHbigenough}). 
    
    When $c_1(X) \neq 0$, quantum cohomology is most generally only
    $\Z/2\Z$-graded, hence it is not necessarily true that $\mathrm{rk}\
    \r{QH}^0(X)_w = 1$ (for instance, if there is only one summand at $w=0$, and
    quantum cohomology is $\Z/2$-graded, then $\mathrm{rk}\ \r{QH}^0(X)_0 =
    \mathrm{rk} \ H^{even}(X)$).  It is easy to find examples in the setting of
    Theorem \ref{thm:mainFano} showing that the hypothesis (\ref{rankinequality})
    cannot be removed.
\end{rem}
The relevant axiomatic properties of Fukaya categories required to prove
Theorems \ref{thm:mainCY} and \ref{thm:mainFano} have been verified in
geometric settings where there are classical constructions available, for
instance for {\em relative Fukaya categories} of pairs $(X^{2n}, \mathbf{D})$
(where $\mathbf{D}$ is an ample simple normal crossings divisor representing
$PD([\omega])$), or for {\em monotone Fukaya categories} of monotone symplectic
manifolds. The proof of the most important (and newest) property we use, Theorem
\ref{thm:ocpairing} below, is a forthcoming joint result of the author, T.
Perutz and N. Sheridan \cite{GPS2:2015} and separately a result of
Abouzaid-Fukaya-Oh-Ohta-Ono \cite{afooo}; for completness we give a
self-contained proof in the technically simplest (i.e., monotone or tautologically
unobstructed) case in Appendix \ref{sec:pairing}.
It is the author's understanding that a version of all of these axiomatic properties for
general Fukaya categories, including Theorem \ref{thm:ocpairing}, for $\K =
\Lambda^{\C}$ the Novikov field over $\C$,
will appear in \cite{afooo}. Assuming this, there is also a version of our
automatic generation result for general Fukaya categories with identical proof;
see Remark \ref{rem:generalFukayaresult} for a formulation.\footnote{Remark
\ref{rem:cyclicity} discusses variations between the frameworks discussed here
and \cite{afooo} as regards {\em Calabi-Yau structures}, which don't affect our
main result.}

\begin{rem} 
In Theorem \ref{thm:mainCY}, the ground field $\K$ most generally needs to be a
version of the Novikov field.  However, for various restricted classes of
Lagrangians in restricted classes of $(X, \omega)$, one can take $\K$ to be a
smaller field. For example, $\K =   \C( ( q) )$ suffices for relative Fukaya
categories, or $\K =  \C$ or
$\F_p$ suffices in the monotone setting of Theorem \ref{thm:mainFano}; see \S
\ref{sec:geometry}. 
\end{rem}

\begin{rem}\label{rem:openmanifolds}
The methods described in this paper require the ambient symplectic manifold $X$
and its Lagrangians to be compact; specifically the structures used require:
\begin{itemize}
    \item finite rank Lagrangian Floer cohomology groups, which is a {\em properness}
condition,
\item a version of Poincar\'{e} duality for Floer cohomology between compact
Lagrangians, 
which is known as a 
{\em weak proper Calabi-Yau structure}, and 

\item a Poincar\'{e} duality non-degenerate self-pairing on the (quantum)
    cohomology of the total space. 
\end{itemize}
Forthcoming work will address related results in wrapped and Landau-Ginzburg
Fukaya categories (each of which have somewhat different hypotheses and are
consequences of rather different geometric structures).
\end{rem}

In a different direction, seeing as all of the geometric structures used have
formal counterparts in any suitable chain-level open-closed topological field theory (in the
vein of e.g., \cite{Costello:2006vn}), it follows that they --- and in particular Theorem
\ref{thm:mainCY} --- have mirror counterparts in algebraic geometry. 
The analogue of Theorem \ref{thm:mainCY} is following generation theorem in
algebraic geometry, which is well known to experts though seems to have no
explicit reference.
\begin{thm}[Folk, Theorem \ref{thmfolkrestate} below]\label{thm:agfolk}
    Suppose $Y$ is a connected smooth and proper Calabi-Yau variety (over $\C$
    for simplicity), and let  $\mc{A} \subset \r{Coh}(Y)$ be a full
    subcategory. If $\mc{A}$ is homologically smooth and non-empty, then
    $\mc{A}$ split-generates $\r{Coh}(Y)$.
\end{thm}
The proof of Theorem \ref{thm:mainCY} can readily be adapted (by considering
the commutative mirror counterparts of the Floer-theoretic structures invoked
here) to prove Theorem \ref{thm:agfolk}. However, there is an even
simpler argument using work of Tabuada \cite{Tabuada:2016aa} that was pointed
out to the author independently by Seidel and Tabuada; see Theorem
\ref{thmfolkrestate} below.

There is an analogous automatic generation theorem for categories of matrix
factorizations on a Calabi-Yau variety mirror to Theorem \ref{thm:mainFano}; we
leave the details to the reader.  

\subsection{Open-closed maps}\label{intro:oc}
Our argument makes essential use of, and has non-trivial consequences for,
geometric {\em open-closed} and {\em closed-open maps}, which relate the
Hochschild invariants of the Fukaya category to quantum cohomology. To discuss
these in a uniform setting for Calabi-Yau and monotone $X$, we temporarily
suppress the summand decompositions which occur in quantum cohomology and
Fukaya categories in the monotone case; see Remark \ref{rem:summanddecomp} for
how to put these back in.

Recall that there is a geometric {\em open-closed map} from the {\em Hochschild
homology} of the Fukaya category, denoted $\r{HH}_{*-n}(\mc{F}(X))$ to the {\em quantum cohomology} of $X$: 
\[
    \oc: \r{HH}_{*-n}(\mc{F}(X)) \ra \r{QH}^*(X)
\]
The importance of this map has highlighted by Abouzaid \cite{Abouzaid:2010kx},
who showed that any subcategory $\mc{A} \subset \mc{F}(X)$ for which
$\oc|_{\mc{A}} : \r{HH}_{-n}(\mc{A}) \ra \r{QH}^0(X)$ hits the unit $1 \in \r{QH}^0(X)$
in fact split-generates $\mc{F}(X)$ (this is reviewed in \S
\ref{sec:openclosed}). Our argument uses properties of open-closed maps and
Fukaya categories to show that, under the hypotheses as in Theorems
\ref{thm:mainCY}-\ref{thm:mainFano}, Abouzaid's criterion is satisfied.

Following \cite{ganatra1_arxiv}, we call any $\mc{A}$ satisfying Abouzaid's
criterion an {\it essential subcategory} (and say $X$ is {\em non-degenerate}
if it is has an essential collection).  In fact, whenever one can find an
essential subcategory $\mc{A}$, it is understood that one can entirely recover
the quantum cohomology ring:
\begin{thm}[\cite{ganatra1_arxiv} for the wrapped Fukaya category, 
    \cite{GPS1:2015}, \cite{afooo}] \label{thm:nondegenerateOC}
    If $\mc{F}(X)$ has an essential subcategory $\mc{A}$, then the
    open-closed and closed-open maps $\oc: \r{HH}_{*-n}(\mc{F}(X)) \cong \r{HH}_{*-n}(\mc{A}) \ra
    \r{QH}^*(X)$ and $\co: \r{QH}^*(X) \ra \r{HH}^*(\mc{F}(X)) \cong
    \r{HH}^*(\mc{A})$ are both isomorphisms.\footnote{The proof given in
    \cite{GPS1:2015} assumes $\mc{A}$ is smooth, which is sufficient for
    our purposes. But see Theorem \ref{thm:smoothness} and
    \cite{ganatra1_arxiv} for a review of how smoothness of $\mc{A}$ follows
    from essentiality.}
\end{thm}
Theorems \ref{thm:mainCY} and \ref{thm:mainFano} are proven by establishing
that the given category $\mc{A}$ is essential; hence using Theorem
\ref{thm:nondegenerateOC} they imply:
\begin{cor}\label{thm:ociso}
    Suppose the hypotheses of Theorems \ref{thm:mainCY} or \ref{thm:mainFano}
    are satisfied for a given subcategory $\mc{A}$ of the Fukaya category of
    $X$; namely  $\mc{A}$ is homologically smooth, non-empty, and satisfies a
    rank inequality on its Hochschild cohomology if $c_1(X) \neq 0$
    (see Theorem \ref{thm:mainFano}, (\ref{rankinequality})). 
    Then, the closed-open and open-closed maps induce
    isomorphisms \begin{equation}\r{HH}_{*-n}(\mc{A}) \cong \r{HH}_{*-n}(\mc{F}(X)) \cong
        \r{QH}^*(X) \cong \r{HH}^*(\mc{F}(X)) \cong \r{HH}^*(\mc{A}).\end{equation}
\end{cor}

\begin{rem}\label{rem:summanddecomp}
    When $X$ is monotone, the above Corollary should strictly speaking be taken in a given summand 
    corresponding to eigenvalues $w$ of $c_1(X) \star -$. Given such a
    $w$, a version of Abouzaid's generation criterion
    \cites{Ritter:2017aa, Sheridan:2016} says that if $\mc{A} \subset
    \mc{F}(X)_w$ and the open-closed map from $\r{HH}_{*-n}(\mc{A})$ hits the
    projection $1_w$ of the unit onto the generalized $w$-eigenspace
    $\r{QH}^*(X)_w$, then $\mc{A}$ split-generates $\mc{F}(X)_w$. Theorem
    \ref{thm:nondegenerateOC} applies with the same proof. Hence, Corollary
    \ref{thm:ociso}, under the hypothesis of Theorem \ref{thm:mainFano},
     concludes that the open-closed maps between the $w$-summand
     $\r{QH}^*(X)_w$ and the Hochschild invariants of $\mc{F}(X)_w$ resp.
     $\mc{A}$ are isomorphisms.
\end{rem}

\begin{rem}\label{rem:converse}
    In light of Remark \ref{smoothnessnecessary}, our results imply that the
    (smoothness, non-emptiness, size of $\r{HH}^0$) hypotheses on $\mc{A}$ in
    Theorem \ref{thm:mainCY} (respectively Theorem \ref{thm:mainFano} or Remark
    \ref{rem:generalFukayaresult}) are in fact {\em equivalent} to the essentiality of
    $\mc{A}$ (and further to open-closed and closed-open maps being isomorphisms).
\end{rem}

\subsection{Overview}

In \S \ref{sec:geometry}, we review the quantum cohomology ring and versions of
the Fukaya category of compact manifolds. In \S \ref{sec:algebra}, we recall
relevant algebraic structures, most notably the notions of split-generation and
homological smoothness, Hochschild invariants, weak Calabi-Yau structures, and
Shklyarov's categorical (Mukai-type) pairing on the Hochschild homology of any
proper category (the key fact about this pairing being {\em non-degenerate} for
homologically smooth categories is recalled in Proposition \ref{shklnondeg}).
\S \ref{sec:openclosed} reviews various geometric structures, such as
open-closed and closed open maps $\oc$ and $\co$, as well as relevant
properties they should satisfy: the linear
duality between $\oc$ and $\co$ in the presence of a weak proper Calabi-Yau
structure\footnote{While our argument does use in an essential way the weak
proper Calabi-Yau structure on the Fukaya category, our main Theorem does
not require $\oc$ and $\co$ to be linear duals; this is only used for
Corollary \ref{thm:ociso}.}, Abouzaid's generation criterion, and the
compatibility of open-closed maps with the Shklyarov pairing).
Combining these general structures with the homological algebra Lemma of
Shklyarov, stated in Proposition \ref{shklnondeg}, the proofs of the main
Theorems, given in \S \ref{sec:proof}, are very short.  In \S
\ref{sec:applications}, we indicate just a few applications of our result,
though we expect many more.

Appendix \ref{appendix:abouzaidimpliessmooth} gives a short (and different) proof of a
result from \cite{ganatra1_arxiv} establishing that $\mc{A}$ must necessarily be
homologically smooth in order to satisfy Abouzaid's criterion.  Appendix 
\ref{sec:pairing} gives a self-contained proof of Theorem \ref{thm:ocpairing}
(the compatibility of open-closed maps with pairings, due to \cite{GPS2:2015,
afooo}) in the in the simplest (monotone and tautologically unobstructed)
geometric cases.

\section*{Acknowledgments}
I am indebted to the work and ideas of (and helpful
conversations/collaborations with) T.  Perutz and N.  Sheridan, which
contributed in two essential ways to this note: first, with the broad idea, 
appearing in \cite{Perutz:2015aa} that there could be intrinsically checkable
conditions on Fukaya algebras which might sometimes, through the
structure of open-closed maps, imply generation, and second, with the specific
observation, appearing in the joint project \cite{GPS1:2015}, that homological
smoothness of a Fukaya algebra of compact Lagrangians in a compact symplectic
manifold implies injectivity of the open-closed map.  
I would like to thank K. Fukaya for explaining to me the result from
\cite{afooo} that distinct summands in general Fukaya categories map
orthogonally to quantum cohomology (which in turn helped clarify the generation
result in that setting, see Remark \ref{rem:generalFukayaresult}), M. Abouzaid
and A. Efimov for helpful conversations, and P.  Seidel and G. Tabuada for
independently pointing out to me the short(er) proof of Theorem
\ref{thm:agfolk} using Tabuada's work \cite{Tabuada:2016aa}.
I would also like to thank A. Hanlon and C. Woodward for corrections on
an earlier draft.

After the first draft of this paper was complete, I learned that a version of the main
results of this paper have been obtained independently by F. Sanda.

\section{The Fukaya category and quantum cohomology}\label{sec:geometry}
Let $(X^{2n},\omega)$ denote a symplectic manifold. We say $X$ is {\em (positively) monotone} if
$c_1(X) = \lambda [\omega]$ for some $\lambda > 0$ and {\em Calabi-Yau} if
$c_1(X) = 0$.
Let $\K$ denote a coefficient field, and let $G$ be an abelian group.

\subsection{Quantum cohomology}
The (small) quantum cohomology \cites{Ruan:1995aa, McDuff:2004aa} of $X$
\[(\r{QH}^*(X; \K), \star)\] is naturally a $\Z/2N$ graded ring, where $N$ is the
minimal Chern number of $X$ (when $c_1(X) = 0$, $\r{QH}^*(X; \K)$ is $\Z$-graded).
As a vector space $\r{QH}^*(X; \K)$ is equal to the cohomology of $X$ with its
grading collapsed.  Quantum cohomology comes equipped with its cohomological
non-degenerate integration pairing
\begin{equation}\label{eq:qhpairing}
    \begin{split}
        \langle - , - \rangle_X: \r{QH}^*(X; \K) \otimes \r{QH}^*(X; \K) &\ra \K\\
    \langle \alpha, \beta \rangle_X &:= \int_X \alpha \cup \beta;
    \end{split}
\end{equation}
the quantum product $\star$ can be described in terms of the corresponding `three-point functions'
\[
( \alpha, \beta, \gamma )_X := \langle \alpha \star \beta, \gamma \rangle_X,
\]
which are given by counts of rigid $J$-holomorphic spheres in $X$ with three
marked points constrained to lie on Poincar\'{e} dual cycles to $\alpha$,
$\beta$, and $\gamma$ respectively. Notably, a given rigid $J$-holomorphic
sphere $u$ is counted with weight
\begin{equation}
    t^{\omega(u)}
\end{equation}
for a chosen $t \in \K$. We will make a {\em well-definedness} and {\em
convergence assumption} that $(\K, t)$ are chosen so that
\begin{enumerate}
    \item (well-definedness) every $t^{\omega(u)}$ is an element of $\K$; and
    \item (convergence) for any triple $\alpha, \beta, \gamma$, the sum of
        $t^{\omega(u)}$ over all $J$-holomorphic spheres $u$ described above is
        an element of $\K$.
\end{enumerate}
Often the convergence hypothesis can only be guaranteed by picking $t$ to be a
formal variable in a Novikov-type field $\K$,
where then convergence is meant in the adic sense and holds as a consequence of
Gromov compactness.  Examples of $(\K, t)$ satisfying the well-definedness and
convergence assumptions depend on $X$; some examples include:
\begin{itemize}
    \item For any symplectic manifold, one can take $\K$ to be the Novikov
        field 
        \[
            \Lambda_{Nov}^{\C}:=  \{ \sum a_i q^{\lambda_i} | a_i \in \C, \lambda_i \in \R,\ \lim_{i \ra \infty} \lambda_i = +\infty \}
        \]
        and weight $t = q$.

    \item If $X$ is Calabi-Yau (or monotone), and $[\omega]$ is an {\em
        integral} symplectic form, one can take $\K = \C( ( q) )$ with $|q| =
        2c_1(X)$, and weight $t = q$.

    \item If $X$ is (positively) monotone, one can take $\K = \C$ (or $\K$
        another small field, such as $\Z_2$ or $\F_p$), and weight $t = 1$.
\end{itemize}
The count of constant spheres factors in a classical contribution $\int_X
\alpha \cup \beta \cup \gamma$; so quantum cohomology is a deformation of the
usual cohomology ring. The quantum cup product is associative.

When $X$ is not Calabi-Yau, quantum cup product with the first Chern class
$c_1(X) \star -$ gives an interesting degree zero endomorphism of the quantum
cohomology ring, inducing a generalized eigenspace decomposition. 

\subsection{The Fukaya category, schematically}
The Fukaya category, as most generally defined (for general Lagrangians in general symplectic manifolds)
\cites{Fukaya:2009ve, Fukaya:2009qf, afooo} associates to a compact symplectic manifold
$X$ a $\ainf$ category $\mc{F}(X)$ over a version of the Novikov field
$\K = \Lambda_{Nov}^{\C} := \{ \sum_{i} a_i q^{\lambda_i} | a_i \in
\C, \lambda_i \in \R, \lambda_i \ra \infty\}$. Its objects are {\em Lagrangian
branes} (meaning Lagrangians equipped with relative Spin structures, local
systems, and potentially grading data when $c_1(X) = 0$), its morphism spaces are
generated as graded $\K$ vector spaces by intersection points
between Lagrangians, and the differential (and $\ainf$ structure maps) $\mu^k:
\hom(L_{k-1}, L_k) \otimes \cdots \otimes \hom(L_0, L_1) \ra \hom(L_0,
L_k)$ are given by counts of rigid $J-$holomorphic discs $u$, each such disc
$u$'s count weighted by $t^{\omega(u)}$, where $t = q$ in this case. 

By restricting to special classes of Lagrangians in special classes of $(X,
\omega)$ (i.e., monotone Lagrangians in a monotone symplectic manifold), one
can frequently work over ``smaller''
pairs $(\K,t)$ (i.e., $(\C((q) ), q)$ or $(\C, 1)$) as in the case of quantum cohomology; all that is necessary is
once more {\em well-definedness} (each $t^{\omega(u)} \in \K$ for $u$ a disc
between Lagrangians in this restricted class) and {\em convergence} (each
$\ainf$ structure map, and all of the subsequent open-closed structure maps, give
convergent counts in $\K$).  Observe that even in a special (such as monotone)
symplectic manifold, it seems necessary to work with
with $\K = \Lambda_{Nov}^{\C}$ if one does not restrict the class of
Lagrangians.

We remark that generally $\mc{F}(X)$ is
a {\em curved $\ainf$ category}; there is an additional element $\mu^0_L \in
\hom(L,L)$ obstructing the differential $\mu^1_L$ squaring to zero; one can
pass to a genuine $\ainf$ category by considering {\em unobstructed} or {\em
weakly unobstructed} objects (or more generally, objects equipped with {\em weak bounding
co-chains}---see i.e., \cite{Fukaya:2010ac}).

As sketched above, the full Fukaya category requires a fair bit of analytic
{\em virtual} machinery to extract counts $\mu^k$ satisfying $\ainf$ relations,
due to the frequent inability to obtain transversely cut out moduli spaces.
Under suitable geometric hypotheses (many of which, such as monotonicity, overlap with the hypotheses
under which one can work with small $(\K, t)$), there are simpler, more
classical  methods of producing transversally cut out moduli spaces, and hence
the $\ainf$ category structure.  Our exposition will focus on two well
known such constructions, the {\em relative Fukaya category}
\cites{Seidel:2002ys, Seidel:2015ab, Sheridan:2015aa, PerutzSheridan:2015} of an integral symplectic
Calabi-Yau or monotone symplectic manifold (relative a given divisor), and another the {\em monotone
Fukaya category} of a monotone symplectic manifold
\cites{Oh:1993aa, Oh:1993ab, Oh:1995aa, Biran:2014aa, Sheridan:2016} (which in turn are
defined along the lines of exact Fukaya categories as in \cite{Seidel:2008zr}).
For a brief discussion of the general case, see Remark \ref{rem:generalFukayaresult}.

\subsection{The relative Fukaya category of a Calabi-Yau
manifold}\label{subsec:relfuk}

We sketch the definition of relative Fukaya categories as is given in the forthcoming
\cite{PerutzSheridan:2015}.  Let $X$ denote a closed integral symplectic
manifold, and $\mathbf{D}$ a simple normal crossings divisor representing
$[\omega]$. Fix a primitive $\alpha$ for the symplectic form $\omega$ on the
the complement $X \backslash \mathbf{D}$ which is {\em Liouville}, in that the
associated vector field $Z$ is outward pointing along a contact neighborhood of
$\mathbf{D}$.  Objects of the {\em relative Fukaya category of $(X,
\mathbf{D})$} are Lagrangian branes as before
which are exact in $X \backslash \mathbf{D}$, and moreover come equipped with
fixed primitives of $\alpha$. The $\ainf$ structure maps are defined as before,
using a special type of almost complex structure adapted to $\mathbf{D}$, with
each disc $u$ weighted by area $q^{\omega(u) - \alpha(\partial u)}$ (note that
this is now a positive {\em integral} weight).  In the special case that each component of
$\mathbf{D}$ is itself ample, studied in \cite{Sheridan:2015aa}, one counts
discs using an almost complex structure $J$ which preserves each component $D$,
and this weight simply records the (positive) intersection multiplicity of $u$
with $\mathbf{D}$. 
By the notation \[\mathcal{F}(X, \mathbf{D})\] we mean the {\it unobstructed}
(or even weakly unobstructed) sub-category of the Fukaya category; i.e., those
objects 
with $\mu^0$ central, so one can talk about Floer cohomology.

\begin{rem}
    More generally, one should consider the enlargement $\mathcal{F}_{mc}(X,
    \mathbf{D})$ whose objects consist of pairs $(L, b)$ of a potentially obstructed object $L$ equipped with
    a {\it weak bounding co-chains}, in the sense mentioned previously. 
    Our discussion applies verbatim to these larger categories, which also
    have well defined cohomological morphism spaces.
\end{rem}

\subsection{The monotone Fukaya category and monotone quantum cohomology}\label{monotoneFuk}
Suppose instead $X$ is a monotone symplectic manifold; specifically, say $[\omega] = 2\tau c_1$ for
$\tau > 0$.  In this case, the quantum cohomology $\r{QH}^*(X):=\r{QH}^*(X,\K)$ of $X$ with $\K$
coefficients is $\Z/2$ graded\footnote{Rather, it is $\Z/2N$ graded as
mentioned above, but we reduce gradings to $\Z/2$, or whichever $\Z/2k$ grading
we are putting on the Fukaya category.}, and as mentioned earlier the operator
$c_1(X) \star -$ induces a decomposition of $\r{QH}^*(X)$ into generalized
eigenspaces $\r{QH}^*_w(X)$, where $w \in \bar{\K}$ is an element of the
algebraic closure.

A Lagrangian $L\subset X$ is said to be {\em monotone} if there is a constant
$\rho > 0$ so that $\omega(-) = \rho \mu_L(-): H_2(M,L) \ra \R$, 
where $\omega(-)$ denotes symplectic area and $\mu_L(-)$ denotes the {\em
Maslov class}.
The {\em monotone Fukaya category} of $X$ is defined over 
$\C$ (or more generally other small fields $\K$ or rings, for instance fields
with finite characteristic), and has as objects monotone Lagrangians $L$ with
minimal Maslov number $\geq 2$ 
equipped as
before with {\em brane data}: relative Spin structures if
$\mathrm{char\ } \K \neq 2$, $\Z/2$ grading data, and $\C^*$ (resp.
$\bar{\K}^*$)-local system.
Objects can also more generally be equipped with {\it weak bounding co-chains}
(see \cite{Fukaya:2010ac} or \cite{Sheridan:2016} in the monotone case).

\begin{rem}
When the minimal Chern number of $X$ is $N > 2$, and $k$ is a number dividing
$N$, one can consider the $\Z/2k$ graded version of the above story with
objects Lagrangians with minimal Maslov number $\geq 2k$ equipped with $\Z/2k$
grading data.  Our results are identical in this case (with the caveat that
$\r{QH}^*(X)$ should also be thought of as $\Z/2k$ graded in Theorem
\ref{thm:mainFano}).
\end{rem}

Given a monotone Lagrangian with brane data, the {\it count of Maslov 2
($J$-holomorphic) discs} in $X$ with boundary on $L$ (potentially weighted by
the $\C^*$/$\bar{\K}^*$ local system) associates a numerical quantity
\[
    n_L \in \bar{\K}.
\]
For each $w \in \bar{\K}$, there is a summand of the monotone Fukaya
category 
\[
\mc{F}_w(X)\]
consisting of Lagrangians (or more generally idempotents of Lagrangians) with
$n_L = w$. Within each summand, the $\ainf$ structure is defined as before; the
main point in verifying the $\ainf$ equations in this case is that all
potentially problematic disc bubbles (which do not occur as codimension 1
boundary of higher $\ainf$ moduli spaces, but may obstruct $(\mu^1)^2 = 0$
\cites{Oh:1993aa, Oh:1995aa}) can be counted and in fact cancel.

A fundamental relationship between the categories $\mc{F}_w(X)$ and quantum
multiplication by $c_1(X)$, deducible from studying the {\it cap product}
action of $c_1(X) \in \r{QH}^*(X)$ on $HF^*(L,L)$ (or equivalently from a closed-open
map), valid only in monotone Fukaya categories, is
\begin{prop}[Auroux \cite{Auroux:2007aa}, Kontsevich, Seidel]
    $\mc{F}_w(X)$ is trivial unless $w$ is an eigenvalue of $c_1(X) \star -: \r{QH}^*(X, \K) \ra \r{QH}^*(X, \K)$.
\end{prop}
See also \cite{Sheridan:2016} for an exposition of the above fact (Cor. 2.9 in
{\em loc. cit.}) and also for a more detailed description of the monotone
Fukaya category.

\section{Categorical structures}\label{sec:algebra}
\subsection{Bimodules and properness}
Given a pair of $\ainf$ categories $\cc$ and $\dd$, there is an associated dg
category of $\ainf$ $\cc\!-\!\dd$  bimodules\footnote{We will particularly
consider the case $\cc = \dd$.}, which we denote $\cc\!-\r{mod}\!-\dd$. This
category by now many references
\cites{Seidel:2008cr, Tradler:2008fk, ganatra1_arxiv,
Sheridan:2015ab}, so we will not provide explicit formulae for the objects and
morphisms of this category. Schematically, an $\ainf$ $\cc\!-\!\dd$ bimodule is
a `bilinear functor' $\mc{B}: \cc^{op} \times \dd \ra Chain_{\K}$ to chain complexes;
meaning, for every pair of objects $(V,W) \in \ob \cc \times \ob \dd$, there is
a chain complex $\mc{B}(V,W)$, along with `multiplication maps' coming from the
induced map on morphism spaces: 
\[
    \begin{split}
    \mu^{r|1|s}_{\mc{B}}:\bigg(\hom_{\cc}(V_{r-1},V_r) \otimes \cdots
    \otimes \hom_{\cc}(V_0,V_1) \bigg) & 
    \otimes \bigg(\hom_{\dd}(W_1,W_0) \otimes \cdots \otimes \hom_{\dd}(W_{s},W_{s-1}) \bigg)\\
    &\longrightarrow \hom_{\K}(\mc{B}(V_0,W_0) , \mc{B}(V_r,W_s)),
\end{split}
\]
satisfying equations coming from the $\ainf$ `bilinear functor' relations (see
\cite{Lyubashenko:2015aa} or more recently \cite{Sheridan:2015ab} for the
formal perspective using multilinear functors).
\begin{rem}
The main property of a bimodule we will study, {\em perfection}, is a Morita
invariant notion. Hence the reader should feel free to replace all instances of
`$\ainf$' by `dg' in the formal discussion (one caveat: when $\cc$ and $\dd$
are dg categories, the correct morphism spaces of dg bimodules in
$\cc\!-\r{mod}\!-\dd$ 
are the derived bimodule homomorphisms).
More precisely, any $\ainf$ category can be replaced by an equivalent, hence
Morita equivalent, dg category, and any $\ainf$ bimodule over a dg category can
be replaced by an equivalent dg bimodule (which is a bilinear dg functor to
chain complexes). 
    \end{rem}

Important examples of bimodules for our purposes include:
\begin{itemize}
    \item {\em Yoneda bimodules}: for a pair of objects $(K,L) \in \ob \cc
        \times \ob \dd$, the $K,L$ Yoneda bimodule, denoted $Y_{K,L}$, is the
        tensor product of the left Yoneda module over $K$ with the right Yoneda
        module over $L$, and associates the following chain complex, for $(A,B)
        \in \ob \cc \times \ob \dd$:
        \[
            Y_{K,L}(A,B):= \hom_{\cc}(K,A) \otimes \hom_{\dd}(B,L).
        \]
        Yoneda bimodules are the analog of {\em free bimodules} over
        categories (specifically, in the category of $A\!-\!B$ bimodules, where
        $A$ and $B$ are $\ainf$ algebras, the analogous bimodule is $A
        \otimes B^{op}$).

    \item the {\em diagonal bimodule} $\cc_{\Delta}$ is a $\cc\!-\!\cc$
        bimodule, which as a chain complex is
        \[
            \cc_{\Delta}(A,B) := \hom_{\cc}(B,A).
        \]
        In the case of an $\ainf$ algebra $A$, the diagonal bimodule
        $A_{\Delta}$ is $A$ thought of as a bimodule over itself.

    \item the {\em (linear) dual diagonal bimodule} $\cc_{\Delta}^{\vee}$ is a
        $\cc\!-\!\cc$ bimodule, which as a chain complex is
        \[
            \cc_{\Delta}^{\vee}(A, B) := \hom_{\cc}(B,A)^{\vee}
        \]
        (where the dual is taken with respect to the ground field $\K$).
\end{itemize}
It is natural to ask whether a bimodule, thought of as a bilinear functor,
actually takes values in the subcategory $perf(\K) \subset Chain_{\K}$ of chain
complexes with finite rank cohomology. 
\begin{defn}\label{def:bimodproper}
    A bimodule $\mc{B}$ is {\em proper} if for any pair of objects $(X,Y)$, the
    total cohomology $H^*(\mc{B}(X,Y))$ is finite rank over $\K$.
\end{defn}
\begin{defn}\label{def:proper}
    An $\ainf$ category $\cc$ is {\em proper} if the total cohomology
    $H^*(\hom_{\cc}(X,Y))$ is finite rank over $\K$, or equivalently if its
    diagonal bimodule $\cc_{\Delta}$ is proper.  
\end{defn}
It is easy to see that Fukaya categories of compact Lagrangians are always
proper; for instance, the Floer co-chain complex of a pair of transverse
compact Lagrangians $L_0, L_1$ is by definition a vector space whose dimension
is the (finite) set of intersection points of the pair.\footnote{When $L_0$ is
not transverse to $L_1$, $\hom_{\mc{F}(X)}(L_0, L_1)$ is at least
quasi-isomorphic to a chain complex generated by the finitely many
(transverse) intersection points of a perturbation $\tilde{L}_0$ with
$L_1$. In some
technical setups, this is in fact taken as the definition of
$\hom_{\mc{F}(X)}(L_0, L_1)$, for a suitable $\tilde{L}_0$.}

\subsection{Split-generation and perfect bimodules}
For any $\ainf$ category $\cc$, denote by $perf(\cc)$ the split-closed
pre-triangulated envelope of $\cc$. There are multiple ways of constructing
this envelope --- each of which comes equipped with a cohomologically full and
faithful embedding $\cc \hookrightarrow perf(\cc)$ --- but all choices are
quasi-equivalent (see
\cite{Seidel:2008zr}*{\S 4c}, where the notation $\Pi Tw(\cc)$ is often
favored; elsewhere $tw^{\pi} \cc$
sometimes appears\footnote{$tw$ or $Tw$ refers to a particular construction of
such an envelope known as {\em twisted complexes}, with $\pi/\Pi$
superscript or prefix to indicate the split-closure. The notation
$perf(\cc)$ references another construction, known as {\em perfect modules}.}).
For any full sub-category $\mc{X} \subset \cc$, there is a corresponding
cohomologically full and faithful embedding $perf(\mc{X}) \hookrightarrow
perf(\cc)$. We say {\em $\mc{X}$ split-generates $\cc$} if this embedding is a
quasi-equivalence.  Equivalently, in $H^0(perf( \cc))$, every object of $\cc$
should be isomorphic to an object of $perf(\mc{X})$, meaning each object admits
a homologically left invertible morphism into some (finite) complex of objects
from $\mc{X}$.

Applying this definition to categories of bimodules, we have: 
\begin{defn}\label{def:perfectbimodule}
    A $\cc\!-\!\cc$ bimodule $\mc{B}$ is {\em perfect} if, in the category of
    $\cc\!-\!\cc$ bimodules, it is split-generated by Yoneda bimodules.
\end{defn}

\subsection{Homological smoothness}\label{sec:smoothness}

We come to the main algebraic finiteness condition appearing in Theorems
\ref{thm:mainCY} and \ref{thm:mainFano}.
\begin{defn}[\cites{Kontsevich:1999aa, Kontsevich:2009ab}]\label{def:smooth}
    A category $\cc$ is \emph{(homologically) smooth} if its diagonal bimodule
    $\cc_{\Delta}$ is a perfect bimodule, in the sense of Definition
    \ref{def:perfectbimodule}.
\end{defn}

We say (for the purposes of this paper) that $\cc$ and $\dd$ are {\em Morita
equivalent} if there is a quasi-equivalence $perf(\cc) \cong perf(\dd)$ (for
instance, this holds whenever there is an embedding $\cc \subset \dd$ which
split-generates). Morita equivalent categories have quasi-equivalent bimodule
categories and in particular, notions of perfectness coincide, hence:
\begin{prop}\label{smoothmorita}
    Smoothness is a {\em Morita-invariant} notion; that is, if $\cc$ is
    homologically smooth, and $\dd$ is Morita equivalent to $\cc$, then
    $\dd$ is homologically smooth. In particular, if a full subcategory $\mc{X} \subset
    \mc{C}$ split-generates, then $\mc{X}$ is smooth if and only if $\mc{C}$
    is.\qed
\end{prop}
A special case of this definition applies to the case of an ordinary
associative or dg algebra $A$ (which arises when $\cc$ has only one object $X$,
with $\mu^k = 0$ for $k \geq 3$,
by setting $A:= \hom_{\cc}(X,X)$).  In that case, $A$ is {\em homologically
smooth} if $A$ is split generated by $A \otimes A^{op}$ in the category of $A
\otimes A^{op}$ modules, e.g., $A$-bimodules. It is well known for instance
that if $A = R$ is a commutative ring over $\C$, then $A$ is smooth if and only
if $Y = Spec\ R$ is smooth in
the usual geometric sense. Proposition \ref{smoothmorita} implies that in fact
the (geometric) smoothness of $Y$ is equivalent to the smoothness of any Morita
equivalent algebra $A'$, such as the algebra of $n\times n$ matrices
$Mat_{n\times n}(R)$. 

Besides directly verifying the definition of smoothness, there are frequently
geometric and topological means of recognizing homologically smooth algebras
and categories, for instance:
\begin{itemize}
    \item As mentioned above, it is well known that the coordinate ring of an
        affine variety $\C[Y]$ is smooth if and only if $Y$ is geometrically
        smooth (see e.g., \cite{Kontsevich:2009ab}*{Ex. 8.4a}).
        
    \item  More generally, the category of perfect complexes   $Perf(Y)$ on any variety $Y$ is homologically smooth if and only if $Y$ is geometrically smooth \cite{Lunts:2010aa}.

    \item The category of coherent complexes $Coh(Y)$ on a variety $Y$ is always homologically smooth\footnote{Note that $Coh(Y) \cong Perf(Y)$ if and only if $Y$ is smooth.} \cite{Lunts:2010aa}.

    \item Matrix factorization categories $MF(Y,W)$ are always homologically smooth \cites{Dyckerhoff:2011aa, Preygel:2011aa, Lin:2013aa}.

    \item The singular chains on the based loop space  of a space $X$,
         $C_{-*}(\Omega X)$ (cohomologically graded by our conventions), is a dg
        algebra\footnote{This is an
        $\ainf$ rather than dg algebra if one doesn't use {\em Moore
        loops}, but the discussion still applies.} with composition induced by
        concatenation of loops. $C_{-*}(\Omega X)$ is homologically smooth
        whenever $X$ is homotopy equivalent to a finite CW complex (see
        \cite{Felix:1995aa}*{Proposition 5.3} or the more recent
        \cite{Abbaspour:2015aa}*{\S 3.2}).

\end{itemize}

\subsection{Hochschild invariants}
To a (small) cohomologically unital $\ainf$ category $\cc$ over a field $\K$, a
standard purely algebraic construction associates {\it Hochschild homology}
\[
    \r{HH}_*(\cc)
\]
and {\it Hochschild cohomology}
\[
    \r{HH}^*(\cc)
\] 
groups, both Morita invariants of $\cc$ (implicitly over $\K$). The Hochschild
cohomology $\r{HH}^*(\cc)$ is a unital ring
and Hochschild homology $\r{HH}_*(\cc)$ is a module 
over $\r{HH}^*(\cc)$.\footnote{Moreover, there are Gerstenhaber algebra/module
structures, as well as other {\em non-commutative calculus} structures not discussed
here.} We adopt the convention that both $\r{HH}_*(\cc)$ and $\r{HH}^*(\cc)$
are cohomologically graded, so that the product and module
structure maps are degree zero maps.  Explicit formulae for these algebraic
constructions can be found in many places, see i.e., \cites{Seidel:2008cr,
ganatra1_arxiv,Sheridan:2015ab}. Hochschild homology is functorial in $\cc$, meaning
that an $\ainf$ functor $F: \cc \ra \dd$ (or more generally a perfect
$\cc\!-\!\dd$ bimodule) induces a pushforward map $F_*: \r{HH}_*(\cc) \ra
\r{HH}_*(\dd)$. For dg categories, Hochschild homology satisfies a K\"{u}nneth
formulae: $\r{HH}_*(\cc) \otimes \r{HH}_*(\dd)
\stackrel{\sim}{\ra} \r{HH}_*(\cc \otimes \dd)$ (a similar statement can be
made for $\ainf$ categories, modulo a discussion the notion of tensor product for
$\ainf$ categories).

Both groups are special cases of a construction which associates, to a bimodule
$\mc{B}$ over $\cc$, Hochschild homology and cohomology groups $\r{HH}_*(\cc,
\mc{B})$, $\r{HH}^*(\cc, \mc{B})$ (we are using the shorthand $\r{HH}_*(\cc) :=
\r{HH}_*(\cc, \cc_{\Delta})$ and $\r{HH}^*(\cc):= \r{HH}^*(\cc,
\cc_{\Delta})$). Even more generally, they can be thought of as (cohomological)
morphisms and/or tensor products associated to the category of $\cc\!-\!\cc$
bimodules. For instance, we have isomorphisms $\r{HH}^*(\cc, \mc{B}) \cong
H^*(\hom_{\cc\!-\!\cc}(\cc_{\Delta}, \mc{B}))$ and $\r{HH}_*(\cc, \mc{B}) \cong
H^*(\cc_{\Delta} \otimes_{\cc\!-\!\cc}^{ \mathbb L} \mc{B})$.

\subsection{Weak proper Calabi-Yau structures}\label{sec:wpCY}
Let $\cc$ be a {\em proper} $\ainf$ category, in the sense of Definition
\ref{def:proper}. A {\em weak proper Calabi-Yau (wpCY) structure of
dimension $n$} on $\cc$ is a quasi-isomorphism of bimodules
\begin{equation}\label{wpcy}
    \cc_{\Delta} \stackrel{\sim}{\lra} \cc^{\vee}[-n].
\end{equation}
Roughly speaking \eqref{wpcy} is the data of a chain-level map realizing a
Poincar\'{e} duality type isomorphism $H^*\hom_\cc(X,Y) \cong
H^{n-*}\hom_{\cc}(Y,X)$, along with a family of chosen higher homotopies
realizing the homotopy-compatibility of this quasi-isomorphism with $\ainf$
multiplications. 

A wpCY structure induces a linear duality between certain Hochschild
invariants:  
\begin{lem}\label{hhweakcy}
    If $\cc$ has a weak proper Calabi-Yau structure of dimension $n$, then
    $\r{HH}_{n-*}(\cc)^{\vee} \cong \r{HH}^*(\cc)$.
\end{lem}
\begin{proof}
    For any $\cc$ there are canonical identifications  $\r{HH}^{*-n}(\cc,
    \cc^{\vee}) \cong \r{HH}_{n-*}(\cc, \cc_{\Delta})^{\vee} :=
    \r{HH}_{n-*}(\cc)^{\vee}$ (the isomorphism can even be realized on the
    level of chain complexes). Next, the wpCY structure induces an isomorphism
    $\r{HH}^{*-n}(\cc, \cc^{\vee}) \cong\r{HH}^*(\cc, \cc_{\Delta}):=
    \r{HH}^*(\cc)$.
\end{proof}

\subsection{The Shklyarov pairing}
Associated to any proper $\ainf$ category $\cc$ over a field $\K$ is a canonical
pairing on its Hochschild homology, called the {\em Shklyarov pairing}:
\begin{equation}\label{shklyarovpairing}
    \langle - , - \rangle_{Shk}: \r{HH}_*(\cc) \otimes \r{HH}_*(\cc) \ra \K.
\end{equation}
(this is
also sometimes called the {\em Mukai pairing}, as, on a proper variety $X$, it
has been proven \cite{Ramadoss:2008aa} that this pairing on $\r{HH}_*(perf(X))$
coincides with the Mukai pairing on $\r{HH}_*(X)$ \cite{Caldararu:2010aa}).
Roughly, any {\em proper} bimodule $\mc{B}$ (see Def. \ref{def:bimodproper}),
which is a bilinear functor $\cc^{op} \times \cc \ra perf(\K) \subset
Chain_{\K}$, induces a map
$\mc{B}_*: \r{HH}_*(\cc^{op}) \otimes \r{HH}_*(\cc) \ra  \r{HH}_*(\cc^{op} \otimes
\cc) \stackrel{\mc{B}}{\ra} \r{HH}_*(perf(\K)) \stackrel{\sim}{\ra} \K$; the
Shkylyarov pairing is the composition of this map with the natural isomorphism
$\r{HH}_*(\cc^{op}) \cong \r{HH}_*(\cc)$.  When $\cc$ is proper, the diagonal
bimodule $\cc_{\Delta}$ is proper, and the Shklyarov pairing $\langle -, -
\rangle_{Shk}$ is then defined as $(\cc_{\Delta})_*$. For the general theory
above (specifically involving expressions like ``$\cc^{op} \otimes \cc$'') to work
most simply, we can assume that $\cc$ is a dg category, as the resulting
pairing is independent of the quasi-equivalence class of $\cc$ and any $\ainf$
category is quasi-equivalent to a dg category. There are direct methods of
defining this pairing when $\cc$ is an $\ainf$ category, using the notion of
{\em multilinear $\ainf$ functors} \cite{Lyubashenko:2015aa} to resolve issues
with tensor
products. This leads to compact formulae for the pairing for an $\ainf$
category with finite dimensional chain-level morphism spaces;  see
\cite{Sheridan:2015ab} 
(and also work of Abouzaid-Fukaya-Oh-Ohta-Ono \cite{afooo}).

The key consequence of smoothness of $\cc$ is the following result of Shklyarov (we draw
upon the concise discussion in \cite{SeidelEquivariant}*{Lecture 8}).
\begin{prop}[\cite{Shklyarov:2013aa}, where the result is attributed to
    Kontsevich-Soibelman]
    \label{shklnondeg}
If $\cc$ is smooth and proper, then its Hochschild homology is finite rank and
the Shklyarov pairing is a non-degenerate pairing.  
\end{prop}
\begin{proof}[Sketch of proof]
    The main idea is to show that for any perfect bimodule $\mc{B}$, there is a `Chern
    character' map:
    \[
        CH_{\mc{B}}: \K \ra \r{HH}_*(\cc) \otimes \r{HH}_*(\cc^{op}),
    \]
    such that when $\mc{B}$ is also proper, the composition
    \begin{equation}\label{shk:composition}
        \r{HH}_*(\cc) \stackrel{CH_{\mc{B}}(1) \otimes id}{\lra} \r{HH}_*(\cc) \otimes \r{HH}_*(\cc^{op}) \otimes \r{HH}_*(\cc) \stackrel{id \otimes \langle -, - \rangle_{Shk}}{\lra} \r{HH}_*(\cc)
    \end{equation}
    is simply the functoriality map on Hochschild homology induced by the
    convolution functor $- \otimes_{\cc} \mc{B}: perf(\cc) \ra perf(\cc)$
    (using the Morita invariance property $\r{HH}_*(\cc) \cong
    \r{HH}_*(perf(\cc))$. When $\cc$ is smooth and proper, the diagonal
    bimodule $\mc{B} = \cc_{\Delta}$ is perfect and proper, hence the above
    applies. In this case, the overall composition \eqref{shk:composition} is
    the identity and the second map is
    (up to the identification $\r{HH}_*(\cc) \cong \r{HH}_*(\cc^{op})$) $id
    \otimes \langle -, - \rangle_{Shk}$, which immediately implies that
    $\r{HH}_*(\cc)$ is finite dimensional and that $\langle -, - \rangle_{Shk}$
    is non-degenerate on the left.\footnote{This argument is an instance of the
    {\em Snake relation} in topological field theory.}
    To clarify, letting $\sigma = CH_{\cc_{\Delta}}(1) =  \sum_{i=1}^k \alpha_i
    \otimes \beta_i$ thought of as an element of $\r{HH}_*(\cc) \otimes
    \r{HH}_*(\cc)$,
    the above implies that, for any $v \in \r{HH}_*(\cc)$,
    \[
        v = \sum_{i=1}^k \alpha_i \cdot \langle \beta_i, v \rangle_{Shk}.
    \]
    A similar argument on the right or an appeal to symmetry of $\langle - , -
    \rangle_{Shk}$ when $\cc$ is weak Calabi-Yau (see i.e.,
    \cite{Sheridan:2015ab}*{Lem.  5.43}) implies non-degeneracy on the right.

    Finally, the map $CH_{\mc{B}}$ exists for formal reasons: any
    object of a category $\dd$ determines a functor $\K \ra \dd$ and hence a
    map $\K \stackrel{\cong}{\ra} \r{HH}_*(\K) \ra \r{HH}_*(\dd)$. A similar
    construction associates an element of $\r{HH}_*(\dd)$ to any object of
    $perf(\dd)$ using the fact that $\r{HH}_*(perf(\dd)) \cong \r{HH}_*(\dd)$.
    To apply this to our situation, note that the category of perfect bimodules
    over $\cc$ is quasi-isomorphic to $perf(\cc \otimes \cc^{op})$;
    in turn, there is a K\"{u}nneth formula $\r{HH}_*(\cc \otimes \cc^{op})
    \cong \r{HH}_*(\cc) \otimes \r{HH}_*(\cc^{op})$ (again, for simplicity,
    this discussion assumes $\cc$ is dg so there are no issues with tensor
    products).
\end{proof}
\section{Geometric structures}\label{sec:openclosed}
In this section, we recall some formal structures possessed by the Fukaya
category and quantum cohomology, primarily concerning {\em
open-closed maps} relating the Hochschild homology and cohomology of the Fukaya
category with quantum cohomology.
The relevant structures
have been established in a number of settings. and in other technical or
geometric setups could be viewed as axiomatic requirements for our result to
hold.  These structures have formal analogues in open-closed topological
field theory  (in the sense of \cite{Costello:2006vn}), and  can be thought of
as a partial axiomatization the open-closed theory expected
to govern Fukaya categories and quantum cohomology.
As usual we
suppress the summand decompositions that occur in the monotone/non-Calabi-Yau
case from the general discussion (except to highlight differences). So when $X$
is monotone, all instances of $\r{QH}^*(X)$ should be replaced with
$\r{QH}^*(X)_w$, and instances of $\mc{F}(X)$ should be replaced with
$\mc{F}_w(X)$.

\subsection{Weak proper Calabi-Yau structure}\label{sec:fukayawpCY}

The Fukaya category carries a canonical weak proper Calabi-Yau structure, in
the sense of \S\ref{sec:wpCY}. 
 Specifically, the wpCY structure gives a first-order chain-level refinement of
 the Poincar\'{e} duality pairing on Lagrangian Floer cohomology; 
\begin{equation}\label{pdfloer}
    HF^*(L_0, L_1) \cong HF^*(L_1,L_0)^{\vee}[-n].
\end{equation} 
On the chain level, the map $HF^*(L_0, L_1) \otimes HF^*(L_1, L_0) \ra \K[-n]$
comes from counts of (unstable) $J$-holomorphic strips with two inputs, or
equivalently, from counts of $J$-holomorphic discs with two boundary marked
points and one unconstrained interior marked point with fixed cross ratio. The
higher order data in this structure has a similar definition (see
\cite{Seidel:2008zr}*{(12j)}, \cite{Seidel:2010aa}*{Proof of Prop. 5.1} and
more recently, \cite{Sheridan:2016}*{\S 2.8}).

It follows from Lemma \ref{hhweakcy} that on the Fukaya category $\mc{F}$ (or any summand in the monotone
case), there are canonical isomorphisms $\r{HH}_{n-*}(\mc{F})^{\vee} \cong
\r{HH}^*(\mc{F})$.

\begin{rem} \label{rem:cyclicity} 
    There are two possible further refinements of a weak proper Calabi-Yau
    structure. 
    In \cites{Fukaya:2010aa, afooo}, a version of the Fukaya category is
    constructed (for fields $\K$ containing $\R$) which is a {\em (strictly)
    cyclic $\ainf$ category}, meaning that certain correlation functions
    $\langle \mu^k(-, \cdots, -), - \rangle$ are graded cyclically symmetric
    on the chain level, for some perfect pairing $\langle -, -
    \rangle: \hom_{\mc{F}}^*(X, Y) \otimes \hom_{\mc{F}}^{*}(Y, X) \ra \K[-n]$.
    Strictly cyclic $\ainf$ categories in particular possess weak proper
    Calabi-Yau structures (induced by $\langle - , - \rangle$), so our
    arguments still apply.

    Over more general $\K$, such as cases in which the Fukaya category can be
    defined over a field or ring of finite characteristic, it may not be possible
    to guarantee the wpCY structure comes from strictly cyclic $\ainf$
    structure. Instead one can show it is ``cyclic up to homotopy''. More
    precisely, the Fukaya category can be equipped with a {\em (strong) proper
    Calabi-Yau structure} (\cite{Gcircleactions}, c.f. the discussion in
    \cite{GPS1:2015}*{\S 6.2}), in the sense of Kontsevich-Soibelman
    \cite{Kontsevich:2009ab}. Over a field of characteristic zero, a proper
    Calabi-Yau structure determines a quasi-isomorphism to a unique
    (isomorphism class) of cyclic $\ainf$ category
    \cite{Kontsevich:2009ab}*{Thm. 10.7}, so in that case the notions are
    essentially equivalent.
\end{rem}

\subsection{The open-closed and closed-open maps}
There is an open-closed map from the Hochschild homology of the Fukaya category
to quantum cohomology
\begin{equation}\label{eq:oc}
    \oc: \r{HH}_{*-n}(\mc{F}) \ra \r{QH}^*(X).
\end{equation}
There is a complementary closed-open map, from the quantum cohomology of $X$ to
the Hochschild cohomology of the Fukaya category:
\begin{equation}\label{eq:co}
    \co: \r{QH}^*(X) \ra \r{HH}^{*}(\mc{F}).
\end{equation}
\begin{rem}
    In the monotone case, it is known by work of Ritter and Smith
    \cite{Ritter:2017aa} (see also \cite{Sheridan:2016}) that both $\oc$ and
    $\co$ respect eigen-summand summand decompositions; namely $\co$
    restricted to $\r{QH}^*(X)_w$ is zero except for the portion mapping to
    $\r{HH}^*(\mc{F}_w(X))$ and $\oc$ from $\r{HH}_{*-n}(\mc{F}_w(X))$ has
    image in $\r{QH}^*(X)_w$.
\end{rem}
In either case, the maps are associated to counts of discs with boundary marked
points and one interior marked point, constrained to go through a (pseudo- or
Morse) cycle in the ambient manifold. In \eqref{eq:oc}, the interior marked
point is the sole output, whereas in \eqref{eq:co}, one of the boundary marked
points is the sole output.  It may seem from this that the structure
coefficients of either map should be therefore identical on the chain level;
however, recall that to define counts
one has to potentially perturb the defining equations in a coherent way
compatible with boundary strata operations. These perturbation schemes and their
compatibility conditions are in general quite different for $\oc$ and $\co$.
Instead, for a compact manifold, $\oc$ and $\co$ can be related cohomologically
as follows.

First, dualizing $\oc$, one obtains a map
\begin{equation}\label{eq:ocvee}
    \begin{split}
        \oc^{\vee}: \r{QH}^*(X)^{\vee} &\ra \r{HH}_{*-n}(\mc{F})^{\vee}\\
        \alpha &\mapsto \langle \alpha, \oc( - ) \rangle_{X}.
    \end{split}
\end{equation}
Using the pairing in quantum cohomology to map $\r{QH}^{*-2n}(X) \ra
\r{QH}^*(X)^{\vee}$, we obtain a map: 
\begin{equation}\label{eq:ocvee1}
    \begin{split}
        \oc^{\vee}_1: \r{QH}^{*}(X) &\ra \r{HH}_{*+n}(\mc{F})^{\vee}\\
        \oc^{\vee}_1(\alpha) &:= \oc^\vee \circ \langle \alpha, - \rangle := \langle \alpha, \oc(-) \rangle
\end{split}
\end{equation}
Since $\mc{F}$ has a weak proper Calabi-Yau structure, Lemma \ref{hhweakcy} provides an
isomorphism $\phi: \r{HH}_{*+n}(\mc{F})^{\vee} \ra \r{HH}^*(\mc{F})$. The desired relationship, which roughly comes from a ``deformation of
perturbation data argument," is then:
\begin{prop}[\cites{afooo, Sheridan:2016, PerutzSheridan:2015}]\label{occodual}
    As cohomology level maps, $\phi \circ \oc^{\vee}_1 \cong \co: QH^*(X) \ra
    \r{HH}^*(\mc{F})$. 
\end{prop}
Namely, modulo identifications of codomains and domains, $\oc$ and $\co$ are linear dual maps.
\begin{rem} 
    In the strictly cyclic framework of \cite{afooo}, $\oc$ and $\co$ can be
    set up to be linear dual on the chain level.  
\end{rem}
\begin{cor}\label{occorelated}
    $\oc$ is an isomorphism if and only if $\co$ is.
\end{cor}

Finally, a well known aspect of compatibility of $\oc$ and $\co$ with algebraic structures is 
\begin{prop}[\cites{Seidel:2002ys, ganatra1_arxiv, Ritter:2017aa}]
    \label{ocmodule}
    $\co$ is an algebra homomorphism, and with respect to the $\co$-induced
    module structure of $\r{HH}_{*-n}(\mc{F})$ over $QH^*(X)$, $\oc$ is
    a $QH^*(X)$ module homomorphism.
\end{prop}

\subsection{Abouzaid's split-generation criterion}
In \cite{Abouzaid:2010kx}, Abouzaid introduced a criterion, in terms of the
open-closed map, for when a collection of Lagrangian branes split-generates the
Fukaya category. Though the original criterion was written for the wrapped
Fukaya category of a Liouville manifold, the statement and its proof are
essentially the same in other contexts (with some modifications to account for
eigenvalue decompositions in the monotone case). This criterion has been
implemented in the monotone setting in work of Ritter-Smith
\cite{Ritter:2017aa} and Sheridan \cite{Sheridan:2016}, and will be
implemented for relative Fukaya categories of Calabi-Yau or monotone manifolds
by Perutz-Sheridan \cite{PerutzSheridan:2015}.

\begin{thm}[Generation criterion, \cite{Abouzaid:2010kx}] \label{thm:gencriterion}
    Let $\bb \subset \mc{F}(X)$ be a full subcategory. If the map
    $\oc|_{\bb}: \r{HH}_{-n}(\bb) \ra QH^0(X)$ contains the unit $e$ in its
    image, then $\bb$ split-generates $\mc{F}(X)$.
\end{thm}

In the monotone case, while the above Theorem suitably interpreted is valid
without passing to summands, it is more useful to have a version for a single
summand of the Fukaya category at a time.  Fix some field $\K$ and grading
group $G$ and consider quantum cohomology and the Fukaya category with $G$
grading over $\K$.
\begin{thm}[Generation criterion, monotone variant, \cites{Ritter:2017aa, Sheridan:2016}]\label{thm:gencriterionmonotone}
    Let $w$ be a generalized eigenvalue of $c_1(X) \star -$, and denote by
    $e_w$ the projection of the unit $e$ to the generalized $w$ eigenspace
    $QH^0(X)_w$. Let $\bb \subset \mc{F}_w(X)$ be a full subcategory of the
    monotone Fukaya category summand corresponding to $w$. If the map
    $\oc|_{\bb}: \r{HH}_{-n}(\bb) \ra QH^0_w(X)$ hits $e_w$, then $\bb$
    split-generates $\mc{F}_w(X)$.
\end{thm}

\subsection{Compatibility with pairings}
The newest feature of the open-closed map $\oc$ which is crucial to our argument
is the compatibility of $\oc$ with pairings:
\begin{thm}[\cite{GPS2:2015}, see also \cite{afooo}]\label{thm:ocpairing}
    Up to a sign of $(-1)^{n(n+1)/2}$, $\oc$ is an {\em isometry}, meaning that
    it intertwines pairings: 
    \[
        \langle \oc(\alpha), \oc(\beta) \rangle_{X} = (-1)^{n(n+1)/2} \langle \alpha, \beta \rangle_{Shk}.
    \]
\end{thm}
This Theorem, which is a version of the Cardy condition arising from a certain
comparison of (operations associated to) different degenerations of
$J$-holomorphic maps from an annulus, will be implemented in
\cite{GPS2:2015} in the setting of the relative Fukaya category. 
See \S
\ref{sec:pairing} for a proof in the monotone and tautologically unobstructed
cases (included here for completeness). In the abstract setting of open-closed
topological field theories (which do not directly apply here), the
compatibility of a formal analogue of the open-closed map with certain pairings
is special case of a result of Costello (see \cite{Costello:2006vn}, Theorem A
and discussion below it).

\section{Proof of Main Results}\label{sec:proof}
\begin{lem}\label{lem:HHbigenough}
    If $\mc{A}$ is any cohomologically unital non-empty (dg or $\ainf$) category, then
    then $\r{rk}\ \r{HH}^0(\mc{A}) \geq 1$. 
\end{lem}
\begin{proof}
    Since $\mc{A}$ is cohomologically unital, $\r{HH}^*(\mc{A})$ is a
    unital graded algebra, and there is a unital map $\r{HH}^*(\mc{A}) \ra
    H^*(\hom_{\mc{A}}(L,L))$ for any object $L$ of $\mc{A}$. When $\mc{A}$ is
    non-empty, by definition there is an object $L \in \ob \mc{A}$ with
    $H^*(\hom_{\mc{A}}(L,L)) \neq 0$, hence $[id_L] \in
    H^0(\hom_{\mc{A}}(L,L))$ is not zero, thus the unit in $\r{HH}^0(\mc{A})$
    cannot vanish.  
\end{proof}

\begin{cor}\label{cor:HHbigenough}
    If $X$ is connected Calabi-Yau, and $\mc{A} \subset \mc{F}(X)$ is a
    non-empty full subcategory of its $\Z$-graded Fukaya category, then 
    $\r{rk}\ \r{HH}^0(\mc{A}) \geq \r{rk}\ \r{QH}^0(X)$.  
\end{cor}
\begin{proof}
    Note that since we are in a $\Z$-graded setting, $\r{rk}\ \r{QH}^0(X) = \r{rk}\
    \r{H}^0(X) = 1$.  Now (since $\mc{F}(X)$ and hence $\mc{A}$ are
    cohomologically unital) apply Lemma \ref{lem:HHbigenough}.
\end{proof}

Below we prove Theorems \ref{thm:mainCY} and \ref{thm:mainFano} simultaneously,
suppressing the summand decompositions that occur in the latter case. So when
$X$ is monotone, all instances of $\r{QH}^*(X)$ should be replaced with
$\r{QH}^*(X)_w$.
\begin{proof}[Proof of Theorems \ref{thm:mainCY} and
    \ref{thm:mainFano}]
    Suppose $\mc{A} \subset \mc{F}$ satisfies the hypotheses of Theorem
    \ref{thm:mainCY} or \ref{thm:mainFano}. It follows that $\mc{A}$ is a
    smooth and proper category, and in particular, $\r{HH}_*(\mc{A})$ is finite
    dimensional and the Shklyarov pairing
    \[
        \r{HH}_*(\mc{A}) \otimes \r{HH}_*(\mc{A}) \ra \K
    \]
    is {\em non-degenerate} by Proposition \ref{shklnondeg}, meaning the
    induced adjoint map 
    \begin{equation}\label{shkladjoint}
        \begin{split}
            \phi_{Shk}^{L}: \r{HH}_*(\mc{A}) &\ra \r{HH}_{-*}(\mc{A})^{\vee}\\
            \phi_{Shk}^L(\alpha) &:= \langle \alpha, - \rangle_{Shk}
        \end{split}
    \end{equation}
    is an isomorphism.  But, by Theorem \ref{thm:ocpairing},  
    the isomorphism $\phi_{Shk}^L$ fits into the following commutative (up to
    an overall sign of $(-1)^{n(n+1)}$) diagram 
    \begin{equation}\label{maindiagram}
        \xymatrix{ \r{HH}_{*-n}(\mc{A}) \ar[dr]^{\oc} \ar[rr]^{\cong}_{\phi_{Shk}^L} & &  \r{HH}_{n-*}(\mc{A})^{\vee}\\
        &\mathrm{QH}^*(X) \ar[ur]^{\oc^{\vee}_1} &}
    \end{equation}
    where $\oc^{\vee}_1$ is as in \eqref{eq:ocvee1}. Note that we are using the
    shorthand $\oc$ above for $\oc|_{\mc{A}}$, and similarly for $\oc^{\vee}_1$.

    It follows immediately from the isomorphism \eqref{shkladjoint} and the
    diagram \eqref{maindiagram} that the map $\oc: \r{HH}_{-n}(\mc{A}) \ra
    \mathrm{QH}^0(X)$ is injective and $\oc^{\vee}_1: \mathrm{QH}^0(X) \ra
    \r{HH}_n(\mc{A})^{\vee}$ is surjective.  Moreover, as all of the vector
    spaces in \eqref{maindiagram} are finite dimensional, $\oc^{\vee}_1$ and
    $\oc$ will be isomorphisms if and only if $\r{rk}\ \r{QH}^0(X) \leq \r{rk}\
    \r{HH}_n(\mc{A})^{\vee}$. Since $\mc{A}$, like any full subcategory of the
    Fukaya category, is {\em weak proper Calabi-Yau} (see \S
    \ref{sec:wpCY} and \S \ref{sec:fukayawpCY}), Lemma \ref{hhweakcy} implies that
    $\r{HH}_n(\mc{A})^{\vee} \cong \r{HH}^0(\mc{A})$, so $\oc$ is an
    isomorphism onto $QH^0(X)$ if and only if $\r{rk}\ \r{QH}^0(X)\leq \r{rk}\
    \r{HH}^0(\mc{A})$. This last condition is ensured by hypothesis in the
    monotone case or automatically in the Calabi-Yau case by Lemma
    \ref{cor:HHbigenough}.  So $\oc$ is an isomorphism from
    $\r{HH}_{-n}(\mc{A})$ onto $\r{QH}^0(X)$ and in particular hits the unit.
    By definition, we have shown that $\mc{A}$ satisfies Abouzaid's
    split-generation criterion (Theorem \ref{thm:gencriterion} or
    \ref{thm:gencriterionmonotone}), so it split-generates $\mc{F}$.
\end{proof}

\begin{rem}
    The fact that smoothness of $\mc{A}$ implies that $\oc|_{\mc{A}}$ is
    injective was already observed in \cite{GPS1:2015}.  
\end{rem}

For completeness, we also recall the version of the proof of Corollary
\ref{thm:ociso} appearing in \cite{GPS1:2015} (which requires just the
structures we have already developed): 
\begin{proof}[Proof of Corollary \ref{thm:ociso}]
    Under the hypotheses given, $\mc{A}$ split-generates $\mc{F}$ so
    $\r{HH}_*(\mc{A}) \cong \r{HH}_*(\mc{F})$ by Morita
    invariance of Hochschild homology. The proof of Theorems 1 and 2, which
    invoke the compatibility of $\oc$ with the Shklyarov pairing (Theorem
    \ref{thm:ocpairing}), imply that the map $\oc: \r{HH}_{*-n}(\mc{F}) \ra
    \r{QH}^*(X)$ is injective and hits the unit.
    By the compatibility of $\oc$
    with module structures (see Proposition \ref{ocmodule}), it follows
    that $\oc$ is surjective too.  
    Hence, $\oc$ is an isomorphism.
    Since $\co$ is linear dual to $\oc$ by Corollary \ref{occorelated}, $\co$
    is an isomorphism too.
\end{proof}

\begin{rem}
    \label{rem:generalFukayaresult}
    Let $X$ be a general symplectic manifold. The methods of
    \cites{Fukaya:2009ve, Fukaya:2009qf, afooo} determine from $X$ a {\em curved
    $\ainf$ category} over $\Lambda$, which contains genuine $\ainf$
    subcategories of weakly unobstructed objects (or
    objects with weak bounding co-chains) $\mc{F}_{\lambda}$ for every $\lambda
    \in \Lambda^{> 0}$.
    Unlike the monotone case, the values of $\lambda$ for which
    $\mc{F}_{\lambda}$ is non-trivial are not necessarily eigenvalues of
    $c_1(X) \star -$ (rather, the precise relationship is unknown, c.f.,
    \cite{Seidel:2014ac}*{Remark 5.4}).  The analogous results to Theorems
    \ref{thm:mainCY} and \ref{thm:mainFano} is: if $\mc{A} \subset
    \mc{F}_{\lambda}$ is homologically smooth, then $\r{HH}_{*-n}(\mc{A})$
    injects onto an idempotent summand $q(\r{QH}^*(X))$, $q \in \r{QH}^*(X)$,
    and moreover generates $\mc{F}_{\lambda, q}$, the projection of the Fukaya
    category onto this summand. 
    
    In particular, if one has found a collection of smooth full subcategories
    $\mc{A}_i \subset \mc{F}_{\lambda_i}$ with $\lambda_i$ distinct, and
    $\sum_i \mathrm{rk}\ \r{HH}^0(\mc{A}_i) \geq \mathrm{rk}\ \r{QH}^0(X)$, then the
     $\mc{A}_i$ split-generate the Fukaya category.

    The last assertion is an immediate consequence of the usual proof given
    above along with the fact, which will appear in \cite{afooo}, explained to
    the author by K.  Fukaya, that the images of $\r{HH}_*(\mc{A}_i)$ in
    $\r{QH}^*(X)$ must be orthogonal.
\end{rem}
    
As described in the introduction, there is an algebro-geometric mirror
counterpart to Theorem \ref{thm:mainCY} stated in Theorem \ref{thm:agfolk},
which seems to be well known:
\begin{thm}[Theorem \ref{thm:agfolk}]\label{thmfolkrestate}
    Suppose $Y$ is a connected smooth and proper Calabi-Yau variety (over $\C$
    for simplicity), and let $\mc{A} \subset \r{Coh}(Y) = \r{perf}(Y)$ be a full
    subcategory. If $\mc{A}$ is homologically smooth and non-empty, then
    $\mc{A}$ split-generates $\r{Coh}(Y)$.
\end{thm}
The proof of Theorem \ref{thm:mainCY} given above can be directly translated
into a proof of Theorem \ref{thmfolkrestate}, using suitable versions of
Hochschild-Kostant-Rosenberg (HKR) maps in place of open-closed maps (these are
also known to be compatible with pairings \cite{Markarian:2009aa,
Ramadoss:2008aa}). We leave these details to the interested reader and instead
give a shortened proof (bypassing the need to invoke the compatibility of HKR maps with pairings)
that was pointed out to the author independently by P. Seidel and G. Tabuada:
\begin{proof}[Proof of Theorem \ref{thmfolkrestate}]
    Since $\mc{A}$ is smooth and proper and $\r{Coh}(Y)$ is proper, 
    \cite{Tabuada:2016aa}*{Theorem 4.4} implies that the inclusion $\mc{A}
    \hookrightarrow \r{Coh}(Y)$ is {\em admissible} i.e., induces a
    semi-orthogonal decomposition $\r{Coh}(Y) = perf(\langle \mc{A},
    \mc{A}^{\perp} \rangle)$. Now, since $Y$ is Calabi-Yau, Serre duality
    implies that the semi-orthogonal complement $\mc{A}^{\perp}$ is in fact
    completely orthogonal to $\mc{A}$. Thus $\r{Coh}(Y)$ is Morita equivalent
    to $\mc{A} \coprod \mc{A}^{\perp}$. Now, the Hochschild cohomology of an orthogonal decomposition splits, hence by Morita invariance and the HKR theorem
    \cite{Hochschild:1962aa, Gerstenhaber:1988aa, Swan:1996aa, Yekutieli:2002aa, Lowen:2005aa, Toen:2007fk} 
    we have
    \begin{equation}
         \r{HH}^0(\mc{A}) \oplus \r{HH}^0(\mc{A}^{\perp}) = \r{HH}^0(\mc{A} \coprod \mc{A}^{\perp})  \stackrel{\mathrm{Morita\ inv.}}{\cong} \r{HH}^0(\r{Coh}(Y)) \stackrel{HKR}{\cong} H^0(Y, \mathcal{O}_Y) \cong \C,
    \end{equation}
    where the last equality uses the fact $Y$ is compact connected (note that the first equality could fail if
    $\mc{A}$ and $\mc{A}^{\perp}$ were simply semi-orthogonal: Hochschild
    cohomology need not split for semi-orthogonal decompositions).  Since $\mc{A}$
    is non-empty, Lemma \ref{lem:HHbigenough} implies that
    $\r{HH}^0(\mc{A})$ has rank $\geq 1$.  Hence $\r{HH}^0(\mc{A}^{\perp}) =
    0$, which implies $\mc{A}^{\perp} = 0$. Therefore $\r{Coh}(Y) =
    perf(\mc{A})$ as desired.
\end{proof}

\section{Applications}\label{sec:applications}
We anticipate many applications of this result to computing (derived) Fukaya
categories and proving homological mirror symmetry. In the latter case, the key
point is to exploit algebro-geometric criteria for verifying homological
smoothness discussed in \S \ref{sec:smoothness}; for instance, if $\cc$ is
Morita equivalent to $Coh(Y)$ for some $Y$ then $\cc$ is smooth
\cite{Lunts:2010aa}.

\subsection{Calabi-Yau homological mirror symmetry}

As already noted in \cite{Perutz:2015aa}, Corollary \ref{coreHMSCY}
below implies simplified proofs of full Homological Mirror Symmetry (HMS) for
Calabi-Yau manifolds, for instance for Calabi-Yau hypersurfaces in
projective space \cites{Seidel:2015ab, Sheridan:2015aa}; namely one does not
need to prove by hand that the Lagrangians considered in these examples
split-generate the Fukaya category. More recently, after the initial version of
this paper appeared, Corollary \ref{coreHMSCY} was used to prove
generation  -- and complete the proof of HMS -- for Calabi-Yau hypersurfaces in
weighted projective spaces, or ``Greene-Plesser mirror pairs''
\cite{SheridanSmithGreenePlesser}.  

To formulate this simplification of HMS precisely, we recall the notion of `Core HMS' introduced by Perutz and Sheridan \cite{Perutz:2015aa}:
\begin{defn}[Compare \cite{Perutz:2015aa}]
    Let $(X,\mathbf{D})$ be a pair with $c_1(X) = 0$, $[\mathbf{D}] = \omega$
    as in \S \ref{subsec:relfuk} and $Y$ a smooth scheme over $Spec( \C((q)))$
    with $c_1(Y) = 0$. We say the pair $( (X, \mathbf{D}), Y)$ satisfy {\em Core HMS} if
    there is 
    \begin{itemize}
        \item a full sub-category $\mc{A} \subset \mc{F}(X,\mathbf{D})$; 
        \item a full split-generating sub-category $\mc{B} \subset D^b_{dg} Coh(Y)$, where $D^b_{dg} Coh(Y)$ is the (dg enhanced) derived category of coherent sheaves; and
        \item a quasi-equivalence $\mc{A} \cong \mc{B}$.
    \end{itemize}
\end{defn}
An immediate corollary of our main Theorem \ref{thm:mainCY} is:
\begin{cor}[Compare \cite{Perutz:2015aa}]\label{coreHMSCY}
    Suppose $(X, \mathbf{D})$ and $Y$ as above satisfy Core HMS. Then the subcategory
    $\mc{A}$ split-generates $\mc{F}(X, \mathbf{D})$, and hence there is an
    equivalence of split-closed derived categories \[perf( \mc{F}(X,
    \mathbf{D})) \cong D^b_{dg} Coh(Y).\]
\end{cor}

\begin{proof}[Proof of Corollary \ref{coreHMSCY}]
    The category $D^b_{dg}(Coh(Y))$ is homologically smooth by a theorem of
    Lunts \cite{Lunts:2010aa}, and smoothness is a Morita invariant notion (see Proposition \ref{smoothmorita}).
    Since $\mathcal{A}$ is quasi-isomorphic to a category which split-generates
    $D^b_{dg}(Coh(Y))$, or in other words Morita equivalent to
    $D^b_{dg}(Coh(Y))$, it follows that $\mathcal{A}$ is non-empty and
    homologically smooth.  Hence, it satisfies the hypotheses of Theorem
    \ref{thm:mainCY}.
\end{proof}

\begin{rem} \label{restrictivehypothesis}
The main Theorem of \cite{Perutz:2015aa} further requires the mirror $Y$ to
have `maximally unipotent monodromy.' Corollary \ref{coreHMSCY}, which places no such
restrictions on $Y$, removes the need to check such a property and would 
seem to apply to a more general class of $Y$.
On the other hand, as observed
in \cite{Perutz:2015aa}, maximally unipotent monodromy of $Y$ is an a
posteriori consequence of any homological mirror equivalence in the Calabi-Yau
setting over Novikov fields: because the top quantum product
$[\omega]^{\star n} = 1 + O(q) \neq 0$, an HMS equivalence guarantees that the
mirror to $[\omega]$ in the cohomology of polyvectorfields on $Y$ must have
non-vanishing top power.  The results of \cite{Perutz:2015aa} imply that, under
homological mirror hypotheses involving the relative Fukaya category, the
symplectic form $[\omega]$ is automatically mirror to the Kodaira-Spencer class
$KS(q \partial q)$;  hence  $KS(q\partial q)^{\wedge n} \neq 0$,
which is a form of maximal unipotence of $Y$.

In general for a Calabi-Yau manifold, the Fukaya category, quantum cohomology,
and open-closed maps are expected (and in some cases known) to be {\em convergent}, in the sense
that one can set $q = \lambda \in \C$ and, for small $\lambda$, all of the
structures continue to be well-defined/involve convergent series (c.f.,
\cite{Seidel:2016ac}). In these cases, it is no longer true that the mirror
(for fixed value of $q=\lambda$) will be a family with any sort of monodromy
hypothesis.  Our automatic generation results would however continue to apply,
provided all open-closed structures discussed (and compatibilities thereof) are
convergent.  
\end{rem}

\begin{rem}
    There was no real need for us to to use the relative Fukaya category
    $\mc{F}(X,\mathbf{D})$ above aside from concretely specifying a technical
    setup; one can make the same definition for a full subcategory of the
    general Fukaya category $\mc{F}(X)$ of \cite{afooo}, in which case the
    mirror is defined over $\Lambda_\C$. Provided all of the general structures
    of Fukaya categories discussed here are verified, the same result continues
    to hold.
\end{rem}

\subsection{Fano Mirror Symmetry}\label{subsec:corehms}

Our new Corollary \ref{cor:corehmsfano} below implies a similar simplified
proof of full HMS for non-Calabi-Yau varieties; we focus here on the Fano case
and leave the general case (using the variant of our main Theorem described in
Remark \ref{rem:generalFukayaresult}) to the reader.  Previously studied
examples to which Theorem \ref{thm:mainFano} apply and simplify existing HMS
proofs (by removing the need to check generation of the Fukaya category by
hand) include:
\begin{itemize}
    \item Fano hypersurfaces in projective space, which were considered in
        \cites{Smith:2012aa, Sheridan:2016}. 

    \item Fano toric varieties, considered in \cite{afooo} (see also
        \cites{Cho:2004aa, Cho:2005aa, Cho:2006aa,  Fukaya:2010ac}) (we further
        expect the general not-necessarily-Fano variant of our Theorem as
        mentioned above to simplify the split-generation aspect of the proof of
        HMS for all toric varieties, as considered in \cite{afooo}).
\end{itemize}
\begin{rem}
When $X$ is a Fano toric manifold, the generation of its Fukaya category by the
monotone toric fiber (equipped with various local systems) is also a corollary
of Evans-Lekili's work on generation by free Lagrangian orbits of Hamiltonian
$G$-actions \cite{Evans:2019aa}. 
\end{rem}

Recall that the mirror partner to a Fano (or general non-Calabi-Yau) variety
is a {\em Landau-Ginzburg (LG) model}, which is a pair $(Y,W)$ of a non-compact
Calabi-Yau variety $Y$ over $\C$ equipped with a holomorphic function $W: Y \ra
\C$. To such an LG model one can associate the dg category of {\em matrix
factorizations}
\begin{equation}\label{eq:mf}
    \r{MF}(Y,W)
\end{equation}
as defined in \cites{Eisenbud:1980aa, Orlov:2012aa, Lin:2013aa} 
The category of matrix factorizations \eqref{eq:mf} reflects
the singularities of the fiber of $W$ over $0\in \C$ (specifically, there is an
equivalence of triangulated categories $H^*(\r{MF}(Y,W)) \cong
D^b_{sing}(W^{-1}(0))$, where $D^b_{sing}(Z)$ is the {\em derived category of
singularities of $Z$}, which vanishes if $Z$ is smooth \cite{Orlov:ys}). We use
the notation
\begin{equation}
    \r{MF}_y(Y,W):=\r{MF}(Y,W-y)
\end{equation}
for the category associated to the fiber over $y\in \C$. There is a
corresponding closed sector group (again associated to each $y \in \C$)
\begin{equation}\label{eq:closedsectorMF}
    \hhmf
\end{equation}
(note that when $W$ has isolated singularities, $H^0(\hhmf)$ is simply the {\em
Jacobian ring} of $W-y$).  One version of HMS for a pair $(X, (Y,W))$ where $X$
is Fano posits that there is a Morita equivalence between $\mc{F}(X) = \coprod_w
\mc{F}_w(X)$ and $\coprod_y \r{MF}_y(Y,W)$ (note that in either case, only
finitely many summands are non-empty); HMS can also be studied a
single summand at a time.

Our Theorem \ref{thm:mainFano} simplifies the task of proving HMS for a pair
$(X, (Y, W))$ in roughly the following way: if one has found a collection of
some Lagrangians $\{L_i\}$ in the monotone symplectic manifold $X$, and
compared the $\ainf$ subcategory of those Lagrangians 
with a split-generating subcategory of $\r{MF}(Y,W)$ 
mirror Landau-Ginzburg model $(Y, W)$, then upon verifying one more hypothesis,
these Lagrangians will split-generate the Fukaya category, implying full HMS.
The additional hypothesis, related to the rank inequality in Theorem 
\ref{thm:mainFano}, is a rank comparison between the zeroth quantum cohomology
and the zero-dimensional piece of the closed string group
\eqref{eq:closedsectorMF} (see the last item below): 
\begin{defn}\label{def:fanocorehms}
    Let $X:=(X,\omega)$ be a monotone symplectic manifold, and $(Y,W)$ an LG
    model. We say the pair $(X, (Y,W))$ satisfies {\em Core HMS for summand
    pair $(w,y) \in \C \times \C$} if there is
    \begin{itemize}
        \item a full sub-category $\mc{A} \subset \mc{F}_w(X)$, 
        \item a full split-generating sub-category $\mc{B} \subset \r{MF}_y(Y,W)$, 
        \item a quasi-equivalence $\mc{A} \cong \mc{B}$, and
        \item a vector space isomorphism $\r{QH}^0(X)_w \cong H^0 (\hhmf)$.
    \end{itemize}
\end{defn}
The last condition is really a statement about a classical `closed string'
mirror equivalence (but just on the level of verifying a rank equality of
vector spaces, not an algebra isomorphism).  We note that both sides
are two-periodically graded, as the differential $[W-y, -]$ only preserves the
parity of the grading of a polyvector field.
\begin{cor}\label{cor:corehmsfano}
    Suppose  $(X, (Y,W))$ satisfies {\em Core HMS for summand pair $(w,y)$}.
    Then, $X$ and $(Y,W)$ are homologically mirror for summand pair $(w,y)$. In
    particular, $\mc{A}$ split-generates $\mc{F}_w(X)$, so there is a
    quasi-equivalence
    \[
        perf( \mc{F}_w(X)) \cong \r{MF}_y(Y,W)
    \]
    (and hence also a ring isomorphism $\r{QH}^*(X)_w \cong H^* (\hhmf)$).
\end{cor}

\begin{proof}[Proof of Corollary \ref{cor:corehmsfano}]
The category $MF_y(Y,W)$ is homologically smooth by the results in
\cites{Dyckerhoff:2011aa, Preygel:2011aa, Lin:2013aa}, and smoothness is a
Morita invariant notion (see Proposition \ref{smoothmorita}).
Since $\mathcal{A}$ is quasi-isomorphic to a category which split-generates
$MF_y(Y,W)$, it follows that $\mathcal{A}$ is non-empty and smooth. Since it is
known \cites{Dyckerhoff:2011aa, Preygel:2011aa, Lin:2013aa} that 
\[
    \r{HH}^*(MF_y(Y,W)) \cong H^* (\hhmf),
\]
it follows by Morita invariance of Hochschild cohomology and the last Core HMS
hypothesis that there are vector space isomorphisms $\r{HH}^0(\mc{A}) \cong
\r{HH}^0(\mc{B}) \cong \r{HH}^0(MF_y(Y,W)) \cong H^0 (\hhmf) \cong
\r{QH}^0(X)_w$. The hypotheses of Theorem
\ref{thm:mainFano} are therefore satisfied.
\end{proof}

\appendix
\section{Abouzaid's criterion implies smoothness}\label{appendix:abouzaidimpliessmooth}

The purpose of this appendix is to recall, and give a simplified proof of, a
result from \cite{ganatra1_arxiv} establishing
a converse to our main Theorems \ref{thm:mainCY} and \ref{thm:mainFano} (or
rather, their means of proof via Abouzaid's criterion):
\begin{thm}[\cite{ganatra1_arxiv}] \label{thm:smoothness}
    If a subcategory $\mc{A} \subset \mc{F}$ of the Fukaya category satisfies Abouzaid's criterion
    then $\mc{A}$ is homologically smooth (hence $\mc{F}$, which is Morita
    equivalent to $\mc{A}$ is too).
\end{thm}
The proof of Theorem \ref{thm:smoothness} occurs at the bottom
of \S \ref{subsec:shriekannulus} and involves two ingredients. 
First, we observe there is a purely algebraic criterion for verifying smoothness of an
$\ainf$ category/algebra $\mathcal{A}$ in terms of the {\em inverse dualizing
bimodule} of $\mathcal{A}$, which is given in \S \ref{smoothnessdual}
(specifically Corollary \ref{cor:smoothnesscriterion}).
Then there is a Floer theoretic commutative diagram (a version of an annulus or
Cardy argument) from \cite{ganatra1_arxiv}, recalled in \S
\ref{subsec:shriekannulus} (specifically Theorem \ref{thm:gthesiscardy}) which
can be applied to deduce the smoothness criterion is satisfied (when Abouzaid's
criterion holds).
\begin{rem}
In \cite{ganatra1_arxiv} Theorem \ref{thm:smoothness} is proven by a different
approach (which is somewhat longer but has a priori stronger consequeces):
first it is shown that the hypothesis of Theorem \ref{thm:smoothness} implies
that products of objects in $\mc{A}$ resolve the diagonal Lagrangian (in a
version of the Fukaya category of the product $X^- \times X$).  Pushing forward by a variant of the
(geometric/Floer-theoretic) ``quilted strips'' functor 
of \cite{MWW1} from the Fukaya category of the product to $\mc{A}$-bimodules
(constructed in \cite{ganatra1_arxiv} using domain moduli spaces described by
Ma'u \cite{Mau:2010lq}), it follows\footnote{once one calculates that this
functor sends the diagonal to the diagonal bimodule and product Lagrangians to
Yoneda bimodules} that the diagonal bimodule is split-generated by Yoneda
bimodules, i.e., that $\mc{A}$ is smooth.  The argument given in this Appendix,
which the author was aware of before \cite{ganatra1_arxiv} was written, gives a
method of deducing smoothness more directly (using fewer Floer-theoretic
arguments) if one is not interested in first resolving the diagonal Lagrangian.
In fact, for a compact symplectic manifold, there is a faster path to a
posteriori showing that the diagonal Lagrangian is split-generated by product
Lagrangians: input Theorem \ref{thm:smoothness} and Theorem
\ref{thm:nondegenerateOC} into \cite[Thm. 7.3]{Abouzaid:2010vn}.
\end{rem}

\subsection{Bimodule duals and a criterion for smoothness}\label{smoothnessdual}
% we recall the notion of bimodule
%dual and its relation to smoothness/perfectness. 
In this subsection we use the notion of bimodule dual 
%(of a bimodule over a category $\cc$) 
to articulate in Proposition
\ref{prop:perfectioncriterion} a criterion for a $\cc\!-\!\cc$ bimodule to be perfect, in analogy with 
%This
%criterion is completely analogous to 
the following criterion for vector spaces: ``a vector space $V$ is finite
dimensional if and only if the canonical map $V^* \otimes V \to \hom(V,V)$ has the
the identity linear map in its image''. Specializing to the diagonal bimodule, we obtain
in Corollary \ref{cor:smoothnesscriterion} a criterion for a (dg or $\ainf$
category) $\cc$ to be smooth.

Recall first (see e.g., \cite{Ginzburg:2007fk} for the case of dgas and
\cite{ganatra1_arxiv} for the
case of $\ainf$ categories) that given any
bimodule $\bb$ over an $\ainf$ category $\cc$, one can form its {\em bimodule
dual} $\bb^!$, whose cohomology for a pair of objects $K,L$ can be
described as a cohomological morphism space in the bimodule category:
%Hochschild cohomology group:
\[
    H^*(\bb^!(K,L)) = 
%    \r{HH}^*(\bb, Y_{K,L}) = 
    H^*(\hom_{\cc-\cc}(\bb, Y_{K,L})).
\]
(the bimodule structure comes from the functoriality of this complex as one varies $K,L$).
In the case of a dga $A$, the bimodule dual of a dg bimodule $B$ is
$B^!:=\mathrm{Rhom}_{A^{op}\otimes A}(B, A^{op} \otimes A)$, where one uses the outer
bimodule structure on $A^{op} \otimes A$ to take $\mathrm{Rhom}$; the inner
structure induces a bimodule structre on $B^!$. 

We will be particularly interested in the case $\bb = \cc_{\Delta}$; in this
case we
call the resulting bimodule dual the {\em inverse dualizing bimodule}
\[
    \cc^! := \cc_{\Delta}^!,
\]
and note that its cohomology, for a pair of objects $K,L$, can be described as a
Hochschild cohomology group 
\[H^*(\cc^!(K,L)) = \r{HH}^*(\cc, Y_{K,L}).\]
Just as the linear dual $V^*$ of a vector space $V$ comes with a canonical map
$V^* \otimes V \stackrel{\alpha}{\to} \hom(V,V)$, note there is a canonical map
(c.f., \cite{ganatra1_arxiv}*{(2.373)})
\[
%    H^*(\cc_{\Delta} \otimes_{\cc\!-\!\cc} \cc^!)  = \r{HH}_*(\cc, \cc^!)
%    \stackrel{[\mu]}{\to} \r{HH}^*(\cc, \cc) = H^*(\hom_{\cc}(\cc_{\Delta},
%\cc_{\Delta}).  
    H^*(\bb^! \otimes^{\mathbb L}_{\cc\!-\!\cc} \bb) \stackrel{[\mu_{\bb}]}{\to} H^*(\hom_{\cc\!-\!\cc}(\bb, \bb)).
%    H^*(\hom_{\cc}(\cc_{\Delta},
%\cc_{\Delta})
\]
(in the dga case, this map is the composition $\mathrm{Rhom}_{A\!-\! A}(B,A^{op}
\otimes A) \otimes_{A\!-\! A}^{\mathbb L} B \to \mathrm{Rhom}_{A\!-\! A}(B,A
\otimes_{A}^{\mathbb L} B \otimes_A^{\mathbb L} A) \to
\mathrm{Rhom}_{A\!-\! A}(B,B)$).
This map can be used to give a criterion for bimodule perfection (item (i) below):
\begin{prop}[Compare Lemma 1.4 of \cite{Abouzaid:2010kx} and Prop. 2.6-2.7 of \cite{ganatra1_arxiv} for the case of modules or objects]
    \label{prop:perfectioncriterion}
%    $\cc$ is smooth if and only if 
    The following are equivalent:
    \begin{enumerate}
    \item[(i)] $[\mu_{\mathcal{B}}]$ has the identity element in its image $[id]_{\bb} \in H^*(\hom_{\cc\!-\!\cc}(\bb,
    \bb))$.
    \item[(ii)] $\bb$ is a perfect bimodule.
    \item[(iii)] $[\mu_{\mathcal{B}}]$ is an isomorphism.  
    \end{enumerate}
    \end{prop}
\def\pp{\mathcal{P}}
%The above Proposition is well-known; 
%(compare \cite{Abouzaid:2010kx}*{\S A}); 
%and is a variation on an argument in e.g.,
%\cite{Abouzaid:2010kx}*{\S A}. 
%we recall a sketch: 
\begin{proof}[Sketch]
If $\bb$ is perfect then it is easy to see (c.f., \cite{ganatra1_arxiv}*{Prop.
2.16}) that $[\mu_\mathcal{B}]$ is an isomorphism ---
one computes that the more general map $\pp^! \otimes^{\mathbb
L}_{\cc\!-\!\cc} \bb \stackrel{\mu}{\to} \hom_{\cc\!-\!\cc}(\pp,
\bb)$ is a quasi-isomorphism whenever $\pp$ is a Yoneda bimodule; moreoever such
maps are (contravariantly) functorial in $\pp$ (e.g., behave well with respect to colimits and
retracts), hence remain quasi-isomorphisms when more generally $\pp$ is perfect. So,
(ii) implies (iii), which tautologically implies (i); it remains to show that (i) (our desired
criterion) in fact implies (ii), that $\mathcal{B}$ is perfect.

Suppose (i) holds, i.e., in particular that $[id]_{\bb} = [\mu_{\bb}] ([\tau])$, 
for some class $[\tau]$. Let us first make the simplifying assumptions that (a)
$\cc$ and $\bb$ are both proper, with finite-dimensional chain-level morphism
spaces and (b) $\cc$ has finitely many objects.
%Now, note
%that 
Fixing specific models, note that (with respect to the bar complex model of
derived tensor product, see e.g., \cite{ganatra1_arxiv}*{Def. 2.20}) the chain
complex $\bb^! \otimes^{\mathbb L}_{\cc\!-\!\cc} \bb$ 
%(see e.g.,
%\cite{ganatra1_arxiv}*{Def. 2.20}), 
admits a natural length filtration; we'll refer to the subcomplex of length
$<R$ by $(\bb^! \otimes^{\mathbb L}_{\cc\!-\!\cc} \bb)^{<R}$; seeing
as the bar complex is a direct sum, any representative $\tau$ lies in a
filtered piece $\tau \in (\bb^! \otimes^{\mathbb L}_{\cc\!-\!\cc} \bb)^{<N}$
for some $N$. 

Now there is a canonical inclusion on the level of chain complexes
\[
    (\bb^! \otimes^{\mathbb L}_{\cc\!-\!\cc} \bb)^{<N} \stackrel{\iota}{\to} \hom_{\cc\!-\!\cc}(\bb, (\cc_{\Delta} \otimes_{\cc}^{\mathbb L} \bb \otimes_{\cc}^{\mathbb L} \cc_{\Delta})^{<N} ),
\]
(compare \cite[eq. (2.357)-(2.360)]{ganatra1_arxiv}). With respect to this map,
$[\mu_{\mathcal{B}}]$ can be described as the (cohomology of the) map that applies $\iota$ and then composes with the canonical ``multiplication'' morphism (which by abuse of notation we also call) 
$\mu_{\mathcal{B}}: (\cc_{\Delta} \otimes_{\cc}^{\mathbb L} \bb \otimes_{\cc}^{\mathbb L} \cc_{\Delta})^{<N} \hookrightarrow (\cc_{\Delta} \otimes_{\cc}^{\mathbb L} \bb \otimes_{\cc}^{\mathbb L} \cc_{\Delta}) \stackrel{\mu_{\bb}}{\to} \bb$ (in \cite{ganatra1_arxiv}*{eq. (2.121)} this last $\mu_{\bb}$ is called ``$\mc{F}_{\Delta, left, right}$'').
In particular,
$\iota (\tau)$ is a morphism in the category of bimodules between $\bb$ and
the bimodule $(\cc_{\Delta} \otimes_{\cc}^{\mathbb L} \bb \otimes_{\cc}^{\mathbb L}
\cc_{\Delta})^{<N}$. Now note that by our simplifying assumptions   
%we note that 
this latter bimodule is in fact perfect (namely, since $\cc$ and $\mc{B}$ are proper with
finite-dimensional chain complexes and $\cc$ has finitely many objects, the
latter bimodule is a finite direct sum of finite-dimensional chain complexes
tensored with Yoneda bimodules).
%(it is a finite complex of Yoneda bimodules, by definition).
%, then each of these filtered sub-bimodules are actually {\em perfect
%bimodules} 
%(the fact that this is perfect uses the fact that $\bb$ and
%$\cc$ are proper with finite-dimensional chain-level morphism spaces), 
Now the
fact that $[\mu_{\mathcal{B}}]([\iota(\mathcal{\tau})]) =  [id]_{\bb}$ is equivalent
to the fact that $\iota(\tau)$ is homologically left invertible with
(homological) left inverse $\mu_{\bb}$. Thus we have shown $\bb$ is a summand of a perfect
bimodule, so is itself perfect.

Finally it remains to lift the assumptions that $\cc$ and $\bb$ are proper (or 
have finite-dimensional co-chain models). We omit the details but instead refer
the reader to an identical argument given in
%of these assumptions can be lifted by an identical argument
%as in 
\cite[\S A, Proof of Lemma 1.4]{Abouzaid:2010kx}; the point is that (since the
bar complex for $\bb^! \otimes^{\mathbb L}_{\cc\!-\!\cc} \bb$ is a direct sum)
any representative $\tau$ lives in a subcomplex of the form $\tau \in \bb^!
\otimes V \otimes \bb$, for $V$ a finite-dimensional chain complex, and hence
the homologically left-invertible morphism from $\bb$ to $(\cc_{\Delta}
\otimes_{\cc}^{\mathbb L} \bb \otimes_{\cc}^{\mathbb L} \cc_{\Delta})^{<N}$ always
factors through a (homologically left invertible morphism into a) perfect
bimodule.  
\end{proof}
%and is an isomorphism whenever
%See Appendix @@CITE for a self-contained proof of
%the converse. 
Applying Proposition \ref{prop:perfectioncriterion} to the diagonal bimodule, we obtain our desired criterion:
\begin{cor}\label{cor:smoothnesscriterion}
    $\cc$ is smooth if and only if the canonical map 
    \[ 
        [\mu_{\cc}]: \r{HH}_*(\cc, \cc^!) = H^*(\cc^! \otimes^{\mathbb L}_{\cc\!-\!\cc} \cc_{\Delta}) \to \r{HH}^*(\cc,\cc) = H^*(\hom_{\cc\!-\!\cc}(\cc_{\Delta}, \cc_{\Delta})).\]
    has the identity in its image (if and only if $[\mu_{\cc}]$ is an isomorphism).  
\end{cor}

\subsection{The Calabi-Yau morphism and an annulus argument from \cite{ganatra1_arxiv}}\label{subsec:shriekannulus}

In \cite{ganatra1_arxiv}, a geometrically defined  (closed) morphism of
bimodules was introduced called the {\em weak smooth CY structure map}:
\begin{equation}\label{cymorphism}
    CY_s: \mc{F}_{\Delta} \to \mc{F}^![n]
\end{equation}
inducing an element $[CY_s] \in \r{HH}^n(\mc{F},
\mc{F}^!)$; these maps are defined by counting discs with two outputs and
arbitrarily many inputs interspersed between them.

\begin{rem}\label{wrappedTFTvscompactTFT}
    Strictly speaking, \cite{ganatra1_arxiv} constructed the map $CY_s$ in the
    case of wrapped Fukaya categories, but the construction carries over
    directly to monotone and tautologically unobstructed cases (and is in fact,
    a strict simplification). It is our understanding that
    Abouzaid-Fukaya-Oh-Ohta-Ono study this map for general compact Fukaya
    categories \cite{afooo}.
\end{rem} 
By applying $\r{HH}_*(\mc{F}, -)$ to \eqref{cymorphism}, or (homologically)
equivalently tensoring with the diagonal bimodule $\mathcal{F}_{\Delta}$, there
is an induced map
\[
    \r{CY}_{s*}: H^*(\mc{F}_{\Delta} \otimes_{\mc{F}\!-\!\mc{F}} \mc{F}_{\Delta}) = \r{HH}_*(\mc{F},\mc{F}) \to 
    \r{HH}_{*+n}(\mc{F}, \mc{F}^!) = H^*(\mc{F}_{\Delta} \otimes_{\mc{F}\!-\!\mc{F}} \mc{F}^!);
\]
in turn, as discussed in \S \ref{smoothnessdual}, there is a canonical map
$[\mu_{\mc{F}}]: H^*(\mc{F}_{\Delta} \otimes_{\mc{F}\!-\!\mc{F}} \mc{F}^!) = \r{HH}_*(\mc{F}, \mc{F}^!) \to 
\r{HH}^*(\mc{F}, \mc{F}) = H^*(\hom_{\mc{F}\!-\!\mc{F}}(\mc{F}_{\Delta}, \mc{F}_{\Delta}))$ for any
$\ainf$ category $\mc{F}$. In conjunction we obtain a map 
\begin{equation}\label{hhdowntohhup}
    [\mu_{\mc{F}}] \circ \r{CY}_{s*}: \r{HH}_*(\mc{F}, \mc{F}) \to \r{HH}^{*+n}(\mc{F}, \mc{F}).
\end{equation}
One of the key results in \cite{ganatra1_arxiv} relates the map
\eqref{hhdowntohhup} to the composition of the closed-open and open-closed map%
\begin{thm}[\cite{ganatra1_arxiv}]\label{thm:gthesiscardy}
    The following diagram commutes, up to an overall sign of $(-1)^{n(n+1)/2}$:
    \[
        \xymatrix{ H^*(\mc{F}_{\Delta} \otimes_{\mc{F}\!-\! \mc{F}} \mc{F}_{\Delta}) = \r{HH}_*(\mc{F},\mc{F}) \ar[d]^{\oc} \ar[r]^{\r{CY}_{s*}} & H^*(\mc{F}^! \otimes_{\mc{F}\!-\! \mc{F}} \mc{F}_{\Delta}) = \r{HH}_*(\mc{F}, \mc{F}^!) \ar[d]^{[\mu_{\mc{F}}]} \\
        \r{QH}^{*+n}(X) \ar[r]^{\co\ \ \ \ \ \ \ \ \ \ \ \ \ \ \ } & \r{HH}^{*+n}(\mc{F}, \mc{F}) = H^*(\hom_{\mc{F}\!-\!\mc{F}}(\mc{F}_{\Delta}, \mc{F}_{\Delta}))}
    \]
    The same diagram commutes when replacing $\mc{F}$ by any full subcategory
    $\mc{A} \subset \mc{F}$ (using the restrictions of the same maps).
\end{thm}
In \cite{ganatra1_arxiv}, this result is proven for the wrapped Fukaya
category, but the methods immediately simplify to recover the result monotone
and tautologically unobstructed cases (where as usual, we can work in a fixed summand
in the monotone case; this notation has been suppressed). The argument involves studying
the moduli space of annuli with arbitrary many input marked points on one
boundary component, and arbitrary many inputs and one output on the other,
along with the constraint that the first input of the first collection and the
output of the second collection are real with differing signs, after some
conformal equivalence with $\{ |z| \in [1,R]\}$. This moduli space has two
degenerations as $R \to \infty$ and $1$; the operations associated to the
(compactification of the) limit as $R \to 1$ are $\mu_{\mc{F}} \circ
\r{CY}_{s*}$, the operation associated to the limit as $R \to \infty$ is $\co
\circ \oc$; therefore these two operations are chain homotopic (with chain
homotopy given by the remaining bubblings of discs).  

We now complete the proof of Theorem \ref{thm:smoothness}:
\begin{proof}[Proof of Theorem \ref{thm:smoothness}]
    Suppose $\mc{A}$ satisfies Abouzaid's criterion, so $\oc:
    \r{HH}_*(\mc{A},\mc{A}) \to \r{QH}^{*+n}(X)$ hits 1. Since $\co$ is a ring
    homomorphism, we conclude that $\co \circ \oc$ hits $[id]_{\mc{A}}$; by the
    commutative diagram of Theorem \ref{thm:gthesiscardy}, this implies that
    $[\mu_{\mc{A}}] \circ \r{CY}_{s*}$, and in particular
    $[\mu_{\mc{A}}]$, hits $[id]_{\mc{A}}$. Now Corollary
    \ref{cor:smoothnesscriterion} implies $\mc{A}$ is smooth.
\end{proof}

\section{Open-closed maps and pairings}\label{sec:pairing}
The goal of this appendix is to provide a self-contained proof, in the simplest cases (i.e., monotone
or tautologically unobstructed cases) of Theorem \ref{thm:ocpairing}
(which is due to \cite{GPS2:2015, afooo}): the open-closed map 
\[
    \oc: \r{HH}_{*-n}(\mathcal{F}(X)) \to QH^*(X)
\]
intertwines pairings, up to an overall sign of $(-1)^{n(n+1)/2}$. As is usual,
we will suppress summand decompositions, instead directing readers to Remark
\ref{rem:summanddecomp} or e.g., \cite{Sheridan:2016} for how to put them back
in. This compatibility
follows from a type of annulus or Cardy argument, similar to those that have
appeared elsewhere in the literature in increasing levels of complexity
\cite{Biran:2009aa,Abouzaid:2010kx, ganatra1_arxiv}. In fact the moduli
spaces appearing are identical to those appearing in the annulus argumentfrom
\cite{ganatra1_arxiv} restated in Theorem \ref{thm:gthesiscardy}, up to
a suitable change of some inputs and outputs. 

In particular, we will give two proofs of Theorem \ref{thm:ocpairing}: the
first proof given here deduces the compatibility of $\oc$ with pairings by
first relating the Shkylarov pairing with another map 
$\r{HH}_*(\mathcal{F}) \to \r{HH}_*(\mathcal{F})^\vee$ induced by the (smooth
and proper) Calabi-Yau structures on $\mathcal{F}$, to which one can apply the
annulus argument of Theorem \ref{thm:gthesiscardy} of \cite{ganatra1_arxiv}
(along with the comparison between $\co$ and the dual of $\oc$ given in
Proposition \ref{occodual}).  The second argument, due jointly to the author
with Perutz-Sheridan \cite{GPS2:2015}, more directly appeals to an annulus
argument to relate the two pairings. (These two arguments are related
morally by ``exchanging certain outputs with inputs'', and indeed there is a
``deformation of data`` type argument, in the spirit of Proposition
\ref{occodual} that could be used to genuinely relate the chain homotopies).

\subsection{Chain-level formulae for the Shklyarov pairing}
We need to recall some conventions about $\ainf$ categories (which we did not
need to explicitly describe for the rest of the paper).
Algebraically, we adopt the convention that in an $\ainf$ category $\cc$, one
has composition maps, for each $d \geq 1$
\begin{equation}\label{compositionmaps}
    \mu^d: \hom_\cc (X_{d-1},X_d) \otimes \cdots \otimes \hom_\cc (X_{0},X_1) \ra \hom_\cc (X_0,X_d)
\end{equation}
of degree $2-d$, satisfying the (quadratic) $\ainf$ relations, for each $k>0$:
    \begin{equation} 
    \label{ainfeq} \sum_{i,l} (-1)^{\maltese_i} \mu^{k-l+1}(x_k, \ldots, x_{i+l+1}, \mu^{l}(x_{i+l}, \ldots, x_{i+1}), x_i, \ldots, x_1) = 0.
\end{equation}
with sign
\begin{equation}\label{ainfsign}
    \maltese_i := ||x_1|| + \cdots + ||x_i||.
\end{equation}
where $|x|$ denotes degree and $||x||:= |x| - 1$ denotes reduced degree. We
adopt the convention of referring to all of these operations as $\mu^*$; $d$ is
implicit from the number of inputs fed in.  
Recall that the (cyclic bar model of the) {\em Hochschild chain complex} of $\cc$ is
\[
    \r{CC}_*(\cc):= \r{CC}_*(\mc{C}, \mc{C}) := \bigoplus_{X_0, \ldots, X_s \in \cc} \hom_{\cc}(X_s, X_0) \otimes \hom_{\cc}(X_{s-1}, X_s) \otimes \cdots \otimes \hom_{\cc}(X_0, X_1)
\]
with the differential as in \cite{Abouzaid:2010kx}*{eq. (5.15)} (with the same
sign conventions in particular). If $\cc$ is proper and moreoever has
finite-dimensional morphism spaces on the chain level, there is an explicit
formula for a chain-level lift of the Shklyarov pairing on $\r{CC}_*(\cc)$
(see \cite[Prop. 5.22]{Sheridan:2015ab}, which gives the $\ainf$ generalization
of a formula in the dg case from \cite[\S 3]{Shklyarov:2013aa}) as follows: 
linearly extend the prescription that
associates to primitive Hochschild chains $\alpha = \alpha_r \otimes \alpha_{r-1}
\otimes \cdots \otimes \alpha_1$ and $\beta =  \beta_s \otimes \beta_{s-1}
\otimes \cdots \otimes \beta_1$, the number 
\begin{equation}\label{shkpairingformula}
    \begin{split}
    \langle \alpha, \beta \rangle_{Shk} := \sum_{i, j,p,q} \mathrm{tr}(  c \mapsto (-1)^{\dagger} \mu^*(\alpha_{i}, \ldots,& \alpha_1, \alpha_r, \cdots \alpha_{i+j+1},\\
    &\mu^*(\alpha_{i+j}, \ldots, \alpha_{i+1}, c, \beta_{p}, \ldots, \beta_1, \beta_s, \ldots, \beta_{p+q+1}), \beta_{p+q}, \ldots, \beta_{p+1})).
\end{split}
\end{equation}
To clarify \eqref{shkpairingformula}, if the expression one is applying $\mu$ to
is not composable in $\cc$, we set the summand to be 0, and $c$ represents an
element in the relevant morphism complex in $\cc$. Finally, we note that the sign $\dagger$ can be expressed as
\begin{equation}
    \begin{split}
    \dagger = &(\sum_{m=1}^i ||\alpha_m||) (1 + \sum_{m=i+1}^k ||\alpha_m|| ) + ||\alpha_r|| + ||\alpha_1|| + \cdots +  ||\alpha_{i-1}|| \\
   &+ (\sum_{n=1}^p ||\beta_n||) (1 + \sum_{n=p+1}^r ||\beta_n||) + ||\beta_s|| + ||\beta_1|| + \cdots + ||\beta_{p-1}|| + |c| |\beta| + ||\beta_{p+1}|| + \cdots + ||\beta_{p+q}||.
   \end{split}
\end{equation}
\begin{rem}
Note that our $\ainf$ sign convention differs from that used in
\cite{Sheridan:2015ab}; to obtain the signs used above, one takes the
formulae in \cite{Sheridan:2015ab}*{Prop. 5.22} and applies the translation
procedure described in \cite{Sheridan:2015ab}*{Remark 3.5}.
\end{rem}

\subsection{The Shklyarov pairing via Calabi-Yau maps}\label{sec:smoothcygeometry}
In this Appendix, we sketch the relationship between the Shkylarov pairing and
other geometric duality structures that have been studied on the Fukaya
category (in \cite{ganatra1_arxiv} and relatedly in ongoing work of
Abouzaid-Fukaya-Oh-Ohta-Ono). 

First we observe that as indicated in Lemma \ref{hhweakcy}, the weak proper Calabi-Yau structure
on the Fukaya category induces a map
\[
    CY_{p*}: \r{HH}^*(\mc{F}, \mc{F}) \to \r{HH}^{*-n}(\mc{F}, \mc{F}^{\vee}) = \r{HH}_{*-n}(\mc{F})^{\vee}
\]
(the last equality is explained in Lemma \ref{hhweakcy}). Composing with the
weak smooth structure, or rather \eqref{hhdowntohhup} we obtain a map
\[
    CY_{p*} \circ [\mu] \circ \r{CY}_{s*}: \r{HH}_*(\mc{F}) \to \r{HH}_{*}(\mc{F})^{\vee}.
\]
A ``deformation of data'' argument, whose proof we omit (as it is in the same
spirit as Proposition \ref{occodual}), implies that this map agrees with the
map $\r{HH}_*(\mc{F}) \to \r{HH}_*(\mc{F})^{\vee}$ induced by the Shklyarov
pairing: 
\begin{prop}\label{prop:shkequalscymucy}
    There is a (homological) equality 
    $CY_{p*} \circ [\bar{\mu}] \circ \r{CY}_{s*} = [\alpha \mapsto \langle \alpha, - \rangle_{Shk}]$.\qed
\end{prop}

\begin{proof}[Proof 1 of Theorem \ref{thm:ocpairing}]
    By Proposition \ref{prop:shkequalscymucy}, $[\alpha \mapsto \langle \alpha,
    - \rangle_{Shk}] = CY_{p*} \circ [\bar{\mu}] \circ \r{CY}_{s*}$. Next, by Theorem \ref{thm:gthesiscardy}, $CY_{p*} \circ [\bar{\mu}] \circ \r{CY}_{s*} = (-1)^{n(n+1)/2} CY_{p*} \circ \co \circ \oc$. Finally, by Proposition \ref{occodual}, $CY_{p*} \circ \co \circ \oc= \oc^{\vee}_1 \circ \oc$, (recall $\oc^{\vee}_1: \tau \mapsto \langle \tau, \oc(-)\rangle_X$) so all together $[\alpha \mapsto \langle \alpha,
    - \rangle_{Shk}] = (-1)^{n(n+1)/2} [\alpha \mapsto \langle \oc(\alpha), \oc(-)\rangle_X]$ as desired.
\end{proof}

\subsection{The Fukaya category and open-closed map}
We will use the geometric consruction of Fukaya categories and notation from
\cite{Seidel:2008zr} (stated for exact Lagrangians and exact symplectic
manifolds) and building on it \cite{Sheridan:2016}*{\S 2} (which is stated for monotone
Lagrangians in a monotone symplectic manifold); either also extends verbatim to the
tautologically unobstructed case where the analysis is even simpler (recall that a Lagrangian is tautologically unobstructed if it bounds no $J$-holomorphic discs for some almost-complex structure $J$).
Following \cite{Sheridan:2016}'s description of the map $\oc$, one first
defines a pairing, which we call
\[
    \langle \oc(-), - \rangle: \r{HH}_*(\mc{F}) \times  \r{QH}^*(X) \to \K
\]
as follows: if $\phi \in \r{HH}_*(\mc{F})$ is a Hochschild chain (in the usual
cyclic bar complex) and $\alpha \in \r{QH}^*(X)$, choose a pseudocycle $f$
representing the Poincar\'{e} dual of the class $\alpha$ and set $\langle
\oc(\phi), \alpha \rangle$ equal to the signed count of the zero-dimensional
component of the moduli space of pseudoholomorphic discs with
Lagrangian boundary conditions and intersection point asymptotics specified by
$\phi$ and interior marked point constrained to lie along the pseudocycle $f$
(this number is independent of choice of $f$ representing $PD(\alpha)$).
Now, applying Poincar\'{e} duality in the
$\r{QH}^*(X)$ factor one obtains $\oc$; i.e., for a basis $\{e_i\}$ of $\r{QH}^*(X)$
with Poincar\'{e} dual basis $\{e^i\}$, 
\[
    \oc(\phi) := \sum_{i} \langle \oc(\phi), e_i \rangle e^i.
\]

\subsection{Degeneration of annuli}

Let $\mc{A}_{r,s}$ denote the abstract moduli space of annuli with 
\begin{itemize}
    \item $r$ positive boundary punctures $z_1, \ldots, z_r$ on the outer boundary component (in clockwise order),
        and

    \item $s$ positive boundary punctures $y_1, \ldots, y_s$ on the inner boundary component (in counterclockwise order); such that

    \item the annulus conformally equivalent after forgetting $z_1, \ldots, z_{r-1}$ and $y_1,
    \ldots, y_{s-1}$ to $S = (\{ z \in \C | 1 \leq z \leq R\}, z_r = -1, y_s = R)$ 
    for some $R$,
\end{itemize}
up to reparametrization. For such an annulus, we call the representative $S$ as
above its {\em standard representative}. On the standard representative, the size of the annulus $R$ and the
(angular) positions of  $z_1, \ldots, z_{r-1}$, $y_1, \ldots, y_{s-1}$ on the
give coordinates on $\mc{A}_{r,s}$; with respect to these coordinates we fix
the orientation
\[
    dR \wedge d y_1 \wedge \cdots \wedge dy_{r-1} \wedge dz_1 \wedge \cdots \wedge dz_{r-1}
\]

We form the Deligne-Mumford compactification
$\bar{\mc{A}}_{r,s}$, which has, as codimension-1 boundary, the following loci: 
\begin{align}
    \label{annulistrata1}  \mc{A}_{r,s-d+1} &\times \mc{R}^d,\ d \geq 1 \\
    \label{annulistrata2}  \mc{A}_{r-d+1,s} &\times \mc{R}^d,\ d \geq 1 \\
    \label{annulistrata3}  (\mc{R}^{ (r-a) + 1 + b} &\times \mc{R}^{a + 1 + (s-b)})^{cyc}\\
    \label{annulistrata4}  \mc{R}_{r,1} &\times_{int} \mc{R}_{s,1}
\end{align}
Here $\mc{R}^d$ is the usual moduli space of discs with $d$ input boundary
punctures and one output boundary puncture (the domains appearing in $\ainf$ structure maps),
and $\mc{R}_{r,1}$ is the moduli of discs with $d$ input boundary punctures and
one interior marked point (the domains appearing in $\oc$). To explain the
strata in some more detail:
in \eqref{annulistrata1}, some
consecutive (with respect to the cyclic ordering) sequence of marked points
$y_{i+1}, \ldots, y_{i+d \mathrm{\ (mod\ d)}}$ has collided and bubbled off of the
inner boundary component. In \eqref{annulistrata2}, similarly a collection of
$d$ adjacent marked points has bubbled off of the outer boundary component (and
$R$ is varying freely). Strata \eqref{annulistrata3}, arising in the limit as
$R \to 1$ consists of all possible configurations of two discs glued to each
other along two boundary marked points as follows: let $\vec{y}_{rest} =
(y_{i+1}, \ldots, y_{i+b})$ be any subsequence of $(y_1, \ldots, y_s)$ not
containing $y_s$, and $\vec{y}_{end} = (y_{i+b+1}, \ldots, y_s, y_1, \ldots,
y_i)$ denote the remainder of the marked points (cyclically reordered to begin
after $\vec{y}_{rest}$; similarly let $\vec{z}_{end} = (z_{j+1}, \ldots,
z_{j+a})$ be a subsequence not containing $z_r$ and let
$\vec{z}_{end} = (z_{j+a+1}, \ldots, z_r, z_1, \ldots, z_j)$ the remaining
points (with the given cyclic reordering). Then, one of  the discs contains
positive marked points labeled by  $(\vec{z}_{end}, z_{in1}, \vec{y}_{rest})$
followed by negative marked point $z_{out1}$ (in that cyclic order); the other contains positive
marked points labeled $(\vec{z}_{rest}, z_{in2}, \vec{y}_{end})$ followed by
negative marked point $z_{out2}$; finally $z_{out1}$ is (nodally) glued to $z_{in2}$ and
$z_{out2}$ is (nodally) glued to $z_{in1}$.
Finally, \eqref{annulistrata4} describes the (codimension 1) limit in which $R \to \infty$, a
configuration of two discs glued to each other along an interior marked point;
besides this point one disc has $s$ boundary marked points (labeled by $y_i$'s)
and the other has $r$ (labeled by $z_i$'s).

Inductively, following the procedure described in \cite{Seidel:2008zr}, equip
$\mc{A}_{r,s}$ with compatible families of strip-like ends and Floer
perturbation data which are consistent with previous choices made along all
strata above, including the $R = 1$ and $R = \infty$ ends (we note that the
strata consist of either $\mc{A}_{r',s'}$ for smaller $r'$, $s'$ which we have
inductively already made choices for, or components which are part of the
$\ainf$ structure or open-closed map, for which we have made such choices
earlier to to define the relevant operations). With respect to such a
set of choices,
given two cyclically ordered tuples of intersections points between Lagrangian
branes, $\vec{a} =  (a_1, \ldots, a_r)$, $\vec{b} = (b_1, \ldots, b_s)$ we
obtain in the usual fashion a moduli space of maps 
\[\mc{A}_{r,s}(\vec{a}, \vec{b})\] 
from an arbitrary annulus (which modulo
reparamaterization lies in $\mc{A}_{r,s}$) solving Floer's equation (with
respect to the perturbation data) with Lagrangian boundary conditions and
asympotics determined by $\vec{a}$ and $\vec{b}$ on each boundary component
respectively. For generic choices of Floer data, our assumptions on $M$ and $L$
imply that there are no sphere or disc bubbles that appear in the
compactification of the 0 and 1-dimensional moduli spaces, and hence
that the Gromov-Floer bordification of the moduli space
obtained by adding maps from broken domains
\eqref{annulistrata1}-\eqref{annulistrata4} as well the usual semi-stable
gluings of Floer strips to maps whose domain lies in $\mc{A}_{r,s}$ is (a) in
the zero-dimensional case, a finite set and (b) in the 1-dimensional case a
compact manifold with boundary described by rigid elements of the broken maps
added.
Before we count general rigid elements of $\overline{\mc{A}}_{r,s}(\vec{a},
\vec{b})$, we examine the operation associated to boundary stratum
\eqref{annulistrata4}.
\begin{lem}\label{diagonalbordism}
The operation associated to the moduli spaces $\mc{R}_{r,1} \times_{int}
\mc{R}_{s,1}$ as $r$ and $s$ vary is chain homotopic to $\langle \oc(-), \oc(-)
\rangle:   \r{CC}_*(\mc{F}) \otimes \r{CC}_*(\mc{F}) \to \K$.
\end{lem}
\begin{proof}[Sketch]
    This is a standard argument, c.f., \cite{Sheridan:2016}*{Proof of Lemma 2.15}:
    Counting maps associated to the configuration \eqref{annulistrata4} with
    discs connected gives a chain map $F: \r{CC}_*(\mc{F}) \otimes
    \r{CC}_*(\mc{F}) \to \K$ by standard arguments.  The moduli problem
    associated to this nodal configuration is equivalent to the problem of counting
    pairs of maps $u, v$ with the domain of $u$ in $\mc{R}_{r,1}$ and the
    domain of $v$ in $\mc{R}_{s,1}$ with the relevant boundary conditions along
    with an interior incidence condition: the pair
    $(u(z_0), v(z_0))$ is constrained to lie along the diagonal cycle
    $[\Delta]$ in $M \times M$. Now, choose a (pseudocycle)
    cobordism $W$ between $\Delta$ and $\sum f_i \times f^i$, where $f_i$
    are pseudocycles Poincar\'{e} dual to the basis $e_i$ of $\r{QH}^*(X)$
    chosen earlier. Counting pairs of maps $u,v$ with incidence condition along
    $W$ defines as usual a chain homotopy 
    between $F$ and $\langle \oc(-), \oc(-) \rangle_X$.
\end{proof}

\begin{proof}[Proof 2 of Theorem \ref{thm:ocpairing}]
Counting rigid elements of $\overline{\mc{A}}_{r,s}(\vec{a}, \vec{b})$ for
varying $\vec{a}$, $\vec{b}$ 
gives a map
\[
    H: \r{CC}_*(\mc{F}) \otimes \r{CC}_*(\mc{F}) \to \K.
\]
By looking at the boundary of 1-dimensional moduli spaces associated to the
same moduli problem, one deduces in the usual fashion that $H \circ (b \otimes
id + id \otimes b) = F - (-1)^{n(n+1)/2} \langle -, - \rangle_{Shk}$, where
$F$ is the operation assocaited to \eqref{annulistrata4}. Note that the
appearance of the sign $(-1)^{n(n+1)/2}$ in this chain homotopy is explained,
for the case $r=s=1$ in \cite[\S 3.9-3.10]{foooasterisque}; see also \cite[\S
5.3]{Seidel:2014aa}; yet another appearance of this sign associated to a slightly
different degeneration of annuli (in terms of what are inputs and what are
outputs along the nodal points when $R \to 1$) appears in \cite[Lemma 6.8]{Abouzaid:2010kx}. The
remaining signs occuring in $\langle -, - \rangle_{Shk}$ are a consequence of the
usual Koszul reordering conventions and will be omitted here; see e.g.,
\cite{Seidel:2008zr, Abouzaid:2010kx} for more details on such computations.
Now Lemma \ref{diagonalbordism} shows that $F$ is chain homotopic to
$\langle \oc(-), \oc(-) \rangle$, completing the proof.
\end{proof}

%%fakesection: bibliography
%\bibliographystyle{amsxport}
\bibliography{/Users/sheelganatra/Dropbox/References/math_bib}
\bibliographystyle{alpha}

\end{document}